\documentclass[11pt]{article}
\usepackage[english]{babel}
\usepackage{amssymb,amsmath,amsthm}
\usepackage[a4paper,margin=2.54cm]{geometry}
\usepackage{graphicx,tikz,ytableau,tikz-cd,graphics}
\usepackage{appendix}

\theoremstyle{plain}
\newtheorem*{theorem*}{Theorem}
\newtheorem{theorem}{Theorem}[] 
\newtheorem{lemma}[]{Lemma}
\newtheorem{proposition}[]{Proposition}
\newtheorem{corollary}[]{Corollary}
\newtheorem{conjecture}[]{Conjecture}

\theoremstyle{definition}
\newtheorem{definition}[]{Definition}
\newtheorem{remark}[]{Remark}
\newtheorem{example}[]{Example}

\DeclareMathOperator{\Wr}{Wr}
\DeclareMathOperator{\He}{He}
\DeclareMathOperator{\htt}{ht}
\DeclareMathOperator{\odd}{odd}
\DeclareMathOperator{\even}{even}

\usepackage{authblk}
\title{Coefficients of Wronskian Hermite polynomials}

\author[1]{Niels Bonneux}
\author[2]{Clare Dunning}
\author[3]{Marco Stevens}
\affil[1,3]{KU Leuven, Department of Mathematics, Celestijnenlaan~200B box 2400, 3001 Leuven, Belgium. E-mail:~{\tt niels.bonneux@kuleuven.be} and {\tt marco.stevens@kuleuven.be}}
\affil[2]{University of Kent, School of Mathematics, Statistics and Actuarial Science, Canterbury CT2 7FS, United Kingdom. E-mail:~{\tt t.c.dunning@kent.ac.uk}}
\date{\today}               
\providecommand{\keywords}[1]{\textit{\textit{Keywords: }} #1}

\usepackage{hyperref}
\hypersetup{colorlinks, citecolor=black, linkcolor=black}
\AtBeginDocument{} 

\begin{document}
\maketitle

\begin{abstract}\label{abstract}
We study Wronskians of Hermite polynomials labelled by partitions and use the combinatorial concepts of cores and quotients to derive explicit expressions for their coefficients. These coefficients can be expressed in terms of the characters of irreducible representations of the symmetric group, and also in terms of hook lengths. Further, we derive the asymptotic behaviour of the Wronskian Hermite polynomials when the length of the core tends to infinity, while fixing the quotient. Via this combinatorial setting, we obtain in a natural way the generalization of the correspondence between Hermite and Laguerre polynomials to Wronskian Hermite polynomials and Wronskians involving Laguerre polynomials. Lastly, we generalize most of our results to polynomials that have zeros on the $p$-star.
\end{abstract}



\section{Introduction}
\label{sec:introduction}

In this paper we focus on Wronskians of Hermite polynomials from a combinatorial viewpoint. For every \textbf{partition} $\lambda=(\lambda_1,\lambda_2,\dots,\lambda_{\ell(\lambda)})$, i.e., a vector of positive integers such that ${\lambda_1 \geq \lambda_2 \geq \dots \geq \lambda_{\ell(\lambda)}>0}$, we consider the \textbf{Wronskian Hermite polynomial}
\begin{equation}\label{eq:WHP}
	\He_{\lambda}:= \frac{\Wr[\He_{n_1},\He_{n_2},\dots,\He_{n_{\ell(\lambda)}}]}{\Delta(n_{\lambda})}
\end{equation}
where $n_\lambda=(n_1,n_2,\dots,n_{\ell(\lambda)})$, defined by $n_i=\lambda_i+\ell(\lambda)-i$, is the \textbf{degree vector} and ${\Delta(n_\lambda)=\prod_{i<j}(n_j-n_i)}$ is the Vandermonde determinant of this vector. We use the notation~$\He_n$ for the $n^\textrm{th}$ \textbf{Hermite polynomial}, which in our convention is a monic polynomial of degree~$n$ defined by 
\begin{equation}\label{eq:recurrenceHermite}
	\He_n(x)
		=x\He_{n-1}(x)-(n-1)\He_{n-2}(x)
\end{equation}
for $n\geq2$, along with the initial conditions $\He_0(x)=1$ and $\He_1(x)=x$.

These Wronskian Hermite polynomials appear in the theory of exceptional orthogonal polynomials~\cite{Duran-Hermite,GomezUllate_Grandati_Milson-Hermite,GomezUllate_Kamran_Milson,Grandati,Haese-Hill_Hallnas_Veselov,Odake_Sasaki,Quesne} and the related topic of rational extensions of the quantum harmonic oscillator~\cite{Duistermaat_Grunbaum,GomezUllate_Grandati_Milson-krein,Marquette_Quesne,Oblomkov}. These polynomials are well-studied. For example, they fulfil recurrence relations~\cite{Bonneux_Stevens,GomezUllate_Kasman_Kuijlaars_Milson}, the asymptotic behaviour of their zeros is derived in~\cite{Kuijlaars_Milson}, and moreover, via iterating rational Darboux transformations applied to the harmonic oscillator, one obtains appealing identities between (pseudo-)Wronskians involving  Hermite polynomials~\cite{Curbera_Duran,GomezUllate_Grandati_Milson-durfee}. Another place where they appear is in the study of rational solutions of the fourth Painlev\'e equation, and in that setting they are called the generalized Hermite and generalized Okamoto polynomials~\cite{Airault,Clarkson-survey,Clarkson-PIV,Kajiwara_Ohta,Murata,Noumi_Yamada,Novokshenov_Shchelkonogov,Okamoto,Veselov,VanAssche}, which also appear in the rational solutions of the Boussinesq equation via a symmetry reduction to Painlev\'e~IV, see~\cite{Clarkson-Boussinesq,Clarkson-Dowie}. An overview of the properties of the generalized Hermite and generalized Okamoto polynomials as well as the precise associated partitions is given in~\cite{VanAssche}. Moreover, it was recently shown that any Wronskian Hermite polynomial is a rational solution of either the Painlev\'e IV equation itself or one of its higher order analogues~\cite{Clarkson_GomezUllate_Grandati_Milson}. In that paper, the authors use the notion of cyclic Maya diagrams to derive rational solutions of the~$A_{2k}$-Painlev\'e system and conjecture that all rational solutions are captured in such a way. In~\cite{Bonneux_Stevens}, the first and third author established fundamental relations for Wronskian Hermite polynomials in terms of the structure of the \textbf{Young lattice}. The most notable contribution therein was a recurrence relation that generates all Wronskian Hermite polynomials along with two initial conditions~\cite[Theorem~3.1]{Bonneux_Stevens}. Subsequently, these relations were extended by replacing the Hermite polynomials by an arbitrary Appell sequence and an explicit connection with the theory of symmetric functions was made in~\cite{Bonneux_Hamaker_Stembridge_Stevens}. A subclass of Wronskian Hermite polynomials also appear in connection with a certain integrable massless quantum field theory~\cite{Carr_Dorey_Dunning}. 

In this article, we use a combinatorial framework  to investigate the coefficients and zeros of the Wronskian Hermite polynomials. The multiplicity of their zeros is well-known to be a triangular number~\cite{Duistermaat_Grunbaum}. Moreover, Veselov's conjecture, quoted in~\cite{Felder_Hemery_Veselov}, states that the zeros are all simple, except possibly at the origin. The multiplicity $k(k+1)/2$ at the origin can be stated exactly according to the number of odd (denoted $p$), respectively even (denoted $q$), elements in the degree vector $n_\lambda$ as  $(p-q)(p-q+1)/2$. This was observed in \cite{Felder_Hemery_Veselov}, although without proof. Later, this statement was shown to hold true for Wronskians of a given sequence of eigenfunctions of the Schr\"odinger equation provided the sequence is {\it semi-degenerate}~\cite{GarciaFerrero_GomezUllate}. However, it remains an open question as to whether the Hermite setting fulfils this semi-degenerate property. We prove via a combinatorial framework that the Wronskian Hermite polynomial can be factorized as
\begin{equation}\label{eq:factorizationHe}
	\He_\lambda(x)
		= x^{\frac{k(k+1)}{2}} R_\lambda(x^2)
\end{equation}
where $R_{\lambda}(0)\neq0$. Moreover, we give a simple combinatorial interpretation for the integer~$k$. Henceforth we  call $R_\lambda$ the \textbf{remainder polynomial}.

\begin{figure}[t]
	\centering
	\begin{tiny}
		\begin{tikzcd}
		& \ydiagram{2,2,2,1} \arrow[r] & \ydiagram{2,1,1,1} \arrow[rd] & \\
		\ydiagram{4,2,2,1} \arrow[ru] \arrow[rd] & & & \ydiagram{2,1} \\
		& \ydiagram{4,1,1,1} \arrow[ruu] \arrow[r] & \ydiagram{4,1} \arrow[ru] &
		\end{tikzcd}
	\end{tiny}
	\caption{All possible choices of sequentially removing domino tiles from the Young diagram of $(4,2^2,1)$. The process terminates at the core $(2,1)$ and the associated Wronskian Hermite polynomial is $\He_{(4,2^2,1)}(x)=x^3(x^6+x^4-7x^2-35)$.}
	\label{fig:dominoprocess}
\end{figure}
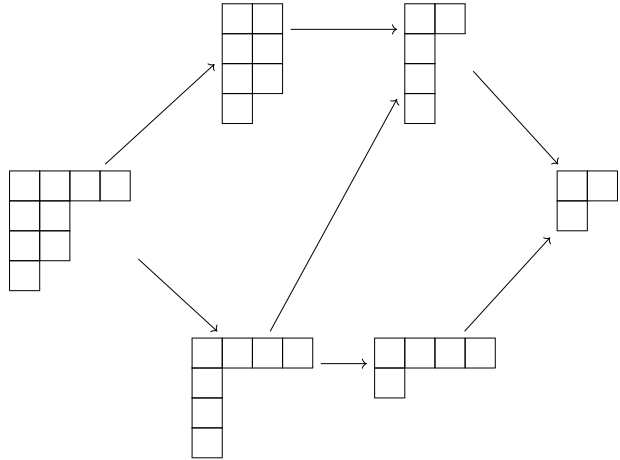

It was proven in~\cite{Bonneux_Hamaker_Stembridge_Stevens} that the coefficients of Wronskian Hermite polynomials, using the convention~\eqref{eq:WHP}, are integers, but explicit values or expressions for the coefficients were not given. We now show that the coefficients and zeros can be understood by considering \textbf{domino tilings} of the \textbf{Young diagram} associated to the partition. The key ingredient is the process of removing rectangles of size 2 (that we naturally refer to as domino tiles), in such a way that the remaining diagram is still a Young diagram of a partition. An example of such a process is given in Figure~\ref{fig:dominoprocess}. Each process terminates when one arrives at a partition of the form $(k,k-1,\dots,2,1)$ for some $k\geq 1$, or the empty partition which we refer to as $k=0$. We prove that the value $k$ is precisely the same $k$ as that specifying the multiplicity of the zero at the origin~\eqref{eq:factorizationHe}. The terminating partition, which always has a staircase Young diagram, is called the \textbf{core} of the original partition~\cite[I.1,~Ex.~8]{MacDonald} and has size $k(k+1)/2$. Hence the result concerning the multiplicity of the zero at the origin has the combinatorial interpretation as the size of the core associated to the partition, see Theorem~\ref{thm:WHdecomposition}. The coefficients of the remainder polynomial $R_\lambda$ can also be expressed within this combinatorial setting and a precise statement is given in Theorem~\ref{thm:coefficients}.

If we let $d$ be the number of domino tiles that one has to remove from the Young diagram of $\lambda$ to obtain its core, we have that $\deg(R_\lambda)=d$. In fact, we show that if $0\leq l \leq d$, then the coefficient of $x^{d-l}$ corresponds to data contained in the partitions one obtains after removing $l$ domino tiles. This process of removing domino tiles from a partition is related to the \textbf{quotient} of the original partition~\cite[I.1,~Ex.~8]{MacDonald}. The quotient is an ordered pair of partitions and they form a lattice isomorphic to the graded lattice $\mathbb{Y} \times \mathbb{Y}$. If $(\mu,\nu)$ is the quotient of the partition~$\lambda$, then the domino process corresponds to all directed paths from $(\mu,\nu)$ to $(\emptyset,\emptyset)$ in $\mathbb{Y} \times \mathbb{Y}$. The example in Figure~\ref{fig:dominoprocess} translates to Figure~\ref{fig:dominoprocessforquotients}.

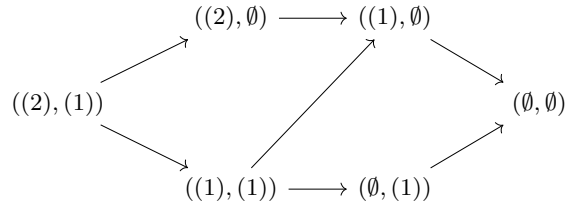
\begin{figure}[t]
	\centering
\begin{footnotesize}
\begin{tikzcd}
	& ((2),\emptyset) \arrow[r] & ((1),\emptyset) \arrow[rd] & \\
	((2),(1)) \arrow[ru] \arrow[rd] & & & (\emptyset,\emptyset) \\
	& ((1),(1)) \arrow[ruu] \arrow[r] & (\emptyset,(1)) \arrow[ru] &
\end{tikzcd}
\end{footnotesize}
\caption{The process of removing domino tiles from the Young diagram of $(4,2^2,1)$ in terms of quotients $(\mu,\nu)$.}
\label{fig:dominoprocessforquotients}
\end{figure}

The combinatorial preliminaries for this article are given in Section~\ref{sec:preliminaries}. In the subsequent sections, we state and prove our main results as described in the following list.
\begin{description}
	\item[Section~\ref{sec:2coresandWHP}.]
	The precise statement of the factorization in~\eqref{eq:factorizationHe} is given in Theorem~\ref{thm:WHdecomposition}.
	
	\item[Section~\ref{sec:coefficients}.] 
	We give several expressions for the coefficients of Wronskian Hermite polynomials.
	\begin{description}
		\item[Section \ref{sec:coeffcharacters}.] Coefficients in terms of characters of irreducible representations of the symmetric group in Theorem~\ref{thm:WHcoeff}.
		\item[Section \ref{sec:coeffhooklengths}.] Coefficients in terms of hook lengths in Theorem~\ref{thm:coefficients}.
		\item[Section \ref{subsec:coeffpoly}.] Coefficients as polynomials in the parameter $k$ in  Theorem~\ref{thm:coeffpoly}.
		\item[Section \ref{sec:coeffsubleading}.] The subleading coefficient of the remainder polynomial in terms of the content in  Proposition~\ref{prop:subleadingcoeff}.
	\end{description}

	\item[Section~\ref{sec:asymptotics}.] 
	We fix the quotient and derive the asymptotic behaviour of the remainder polynomial~$R_\lambda$ when the size of the core $k(k+1)/2$ tends to infinity; see  Theorem~\ref{thm:asymptotics}.
	
	\item[Section~\ref{sec:connectionwithlaguerre}.]
	We identify Wronskian Hermite polynomials with Wronskians involving Laguerre polynomials \cite{Bonneux_Kuijlaars,Duran-Laguerre,Duran_Perez,GomezUllate_Grandati_Milson-L+J}. The obtained identity is naturally labelled by the partition and its core and quotient; see Proposition~\ref{prop:WHermiteWLaguerre}. This generalizes well-known identities that interpret Hermite polynomials in terms of Laguerre polynomials; see \eqref{eq:HermiteLaguerre} below.
	
	\item[Section~\ref{sec:generalization}.] 
	Dominoes -- $2\times1$ or $1\times2$ rectangles -- can be viewed as {\it border strips} of size~2. Removing dominoes from a Young diagram can be generalized to removing border strips of size~$p$. In this way we obtain the $p$-core and the associated $p$-quotient of a partition~\cite[I.1,~Ex.~8]{MacDonald}, with $p=2$ representing the removal of dominoes described in the Hermite setting. Therefore we can generalize most of our results to general integers $p$, as explained in Section~\ref{sec:generalization}. 
\end{description}

As a by-product of our work, we obtain results concerning ratios of hook lengths, see Corollary~\ref{cor:int2} and Corollary~\ref{cor:intp}. Both corollaries also follow from a more general statement about hooks in partitions, see Theorem~4.4 and Remark~4.5 in~\cite{Bessenrodt}, although our argument to arrive at these statements is completely different. Other related results about hook ratios can be found in~\cite{Dehaye}.

Finally, we end with a conclusion including a description of further research possibilities. The appendix contains a combinatorial identity and an extension of Proposition~\ref{prop:subleadingcoeff} to Wronskian Appell polynomials \cite{Bonneux_Hamaker_Stembridge_Stevens}.


\section{Preliminaries}\label{sec:preliminaries}
We introduce the combinatorial concepts of partitions, cores and quotients and related aspects including Sato's Maya diagrams. A standard reference for partitions is~\cite{MacDonald} or \cite{Stanley_EC2} where these concepts (except for Maya diagrams) are clearly explained, and \cite{Sato1,Sato2} for an introduction to Maya diagrams. For some recent advances about cores, we refer to~\cite{Nath} and the references therein. 

\subsection{Partitions and Maya diagrams}
A \textbf{partition} is a vector $\lambda=(\lambda_1,\lambda_2,\dots,\lambda_{\ell(\lambda)})$ of integers such that $\lambda_1\geq \lambda_2 \geq \dots \geq \lambda_{\ell(\lambda)}> 0$. Its \textbf{size} is denoted by $\lvert \lambda \rvert =\lambda_1 + \lambda_2 +\dots + \lambda_{\ell(\lambda)}$ and its \textbf{length} by $\ell(\lambda)$. The \textbf{degree vector} associated to $\lambda$ is given by $n_\lambda=(n_1,n_2,\dots,n_{\ell(\lambda)})$, where $n_i = \lambda_i +{\ell(\lambda)}-i$. We especially note that this implies that $n_1>n_2>\dots>n_{\ell(\lambda)}> 0$. Each partition $\lambda$ can be visualized by its \textbf{Young diagram}
\begin{equation*}
	D_\lambda
		=\{(i,j)\in\mathbb{Z}^2: 1\le i\le\ell(\lambda),\ 1\le j\le\lambda_i\}
\end{equation*}
which consists of $\ell(\lambda)$ rows, and the $i^\textrm{th}$ row has $\lambda_i$ boxes. The points $(i,j)\in D_\lambda$ are often depicted as unit squares with matrix-style coordinates. Clearly, the size of the partition is equal to the number of boxes in its Young diagram. The \textbf{Young lattice} $\mathbb{Y}$ is the set of all partitions partially ordered by inclusion of the corresponding Young diagrams, that is $\mu\leq\lambda$ if $\mu_i \leq \lambda_i$ for all $i=1,2,\dots,\ell(\mu)$. We write $\mu\lessdot\lambda$ or $\lambda\gtrdot\mu$ to indicate that
$\lambda$ covers $\mu$ in $\mathbb{Y}$; that is $\mu<\lambda$ and $|\lambda|-|\mu|=1$. For convenience, within a partition we let $\lambda_i^{t}$ denote $\lambda_i$ repeated $t$ times, and $\emptyset$ denotes the unique partition of zero.

The \textbf{Maya diagram} $M_\lambda$ associated to a partition $\lambda$ is the set
\begin{equation*}
	M_\lambda
		= \{ n \in \mathbb{Z} \mid n<0\} \cup \{n_i \mid 1\leq i \leq \ell(\lambda)\} \subset \mathbb{Z}
\end{equation*}
where  the elements $n_i$ form the degree vector of $\lambda$. This diagram can be visualized by a doubly-infinite sequence of consecutive boxes that are  either filled with a dot or are empty. The boxes are labelled by the integers and the $n^\textrm{th}$ box is filled precisely when $n\in M_\lambda$. Furthermore, a vertical line is placed between the boxes labelled $-1$ and $0$; subsequently we can omit the labels. We can shift the origin $t$ steps (for any positive integer $t$) to the left such that the sequence of filled and empty boxes remains unchanged, but the labelling differs. We call such Maya diagrams, denoted by $M_{\lambda}+t$, equivalent to $M_{\lambda}$ and we refer to $M_{\lambda}$ as the canonical Maya diagram associated to $\lambda$. From the assumption $\lambda_{\ell(\lambda)}>0$ it follows that $M_{\lambda}$ is the unique diagram 
in which the first box to the right of the origin is empty, whereas equivalent Maya diagrams start with filled boxes. This can be interpreted as adding $t$ zeros to the partition. See Figure~\ref{fig:exampleMayaYoung} (left) for an example and observe that the number of filled boxes to the right of the origin is equal to $\ell(\lambda)+t$ for any equivalent Maya diagram. The parts of the partition are read off any Maya diagram by counting the number of empty boxes to the left of each filled box. 

\begin{figure}[t]
\centering
\begin{tikzpicture}[scale=0.5]
	\draw[very thin,color=gray] (-2.5,1) grid (12.5,0);
	\draw[thick,color=black] (0,1.4) -- (0,-0.4);
	\draw (-3,0.5) node {$\dots$};
	\draw (13,0.5) node {$\dots$};
	\draw (-6,0.5) node {$M_{(4^2,2^2,1)}$};
	\foreach \x in {-2,-1,1,3,4,7,8}
	{
		\draw (\x+0.5,0.5) node {$\bullet$};
	}

	\draw[very thin,color=gray] (-2.5,-1) grid (12.5,-2);
	\draw[thick,color=black] (0,-0.6) -- (0,-2.4);
	\draw (-3,-1.5) node {$\dots$};
	\draw (13,-1.5) node {$\dots$};
	\draw (-6,-1.5) node {$M_{(4^2,2^2,1)}+2$};
	\foreach \x in {-2,-1,0,1,3,5,6,9,10}
	{
		\draw (\x+0.5,-1.5) node {$\bullet$};
	}
	\draw (18,-0.5) node 
		{	
			\begin{ytableau}
				8 & 6 & 3 & 2 \\
				7 & 5 & 2 & 1 \\
				4 & 2 \\
				3 & 1 \\
				1
			\end{ytableau}
		};
\end{tikzpicture}
\caption{Left: the canonical and an equivalent Maya diagram associated to the partition $\lambda=(4^2,2^2,1)$. Right: the corresponding Young diagram and its hook lengths.}
\label{fig:exampleMayaYoung}
\end{figure}
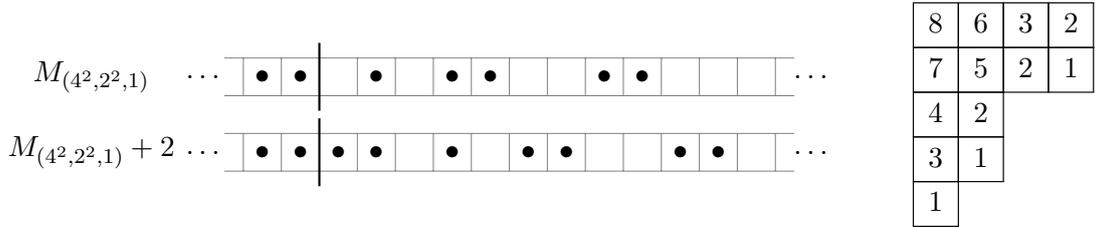

If one reflects a Young diagram in the diagonal, that is rows switch to columns and vice versa, one obtains the Young diagram of the \textbf{conjugate partition} $\lambda'$. In terms of Maya diagrams, conjugation amounts to reflecting the diagram under the mapping $z\mapsto -z-1$, interchanging the filled boxes with the empty boxes and shifting the vertical line so that it lies before the first empty box.

For every box in a Young diagram, the \textbf{hook length} is the number of boxes in the column below the box, plus the number of boxes to the right plus one to account for the box itself. See Figure~\ref{fig:exampleMayaYoung} (right) for an example, where the hook lengths are written inside the boxes. One observes that the hook lengths in the first column of the Young diagram are precisely the elements of the degree vector of the partition.

The number of \textbf{directed paths} (equivalently \textit{saturated chains}) in the Young lattice between the empty partition $\emptyset$ and a partition $\lambda$ is denoted by $F_\lambda$. It is also the number of standard Young tableaux of shape $\lambda$ and it equals the dimension of the irreducible representation of the symmetric group associated to $\lambda$. This number has several useful expressions in terms of partition data.  For this, we write $h_{i,j}$ for the hook length of the $j^\textrm{th}$ box in the $i^\textrm{th}$ row and define $H(\lambda):=\prod_{(i,j)\in \lambda} h_{i,j}$. Furthermore, we let $\Delta(n_{\lambda})$ be the Vandermonde determinant of the degree vector $n_\lambda$, that is $\Delta(n_{\lambda})=\prod_{i<j}(n_j-n_i)$. Then
\begin{equation}\label{eq:Flambdahooklengths}
	F_\lambda 
		= \frac{\lvert \lambda \rvert!}{H(\lambda)}
	\qquad
	\qquad
	H(\lambda)
		= \frac{\prod_{i} n_i!}{|\Delta(n_{\lambda})|}
\end{equation}
for every partition $\lambda$. Further, for any pair of partitions $\tilde{\lambda}\leq \lambda$ we write 
$F_{\lambda/\tilde{\lambda}}$ for the number of paths from $\tilde{\lambda}$ to $\lambda$. In particular, $F_{\lambda/\emptyset}=F_{\lambda}$.

The set $\mathbb{Y}\times\mathbb{Y}$ is a graded lattice, when using the ordering $(\tilde{\mu},\tilde{\nu})\leq(\mu,\nu)$, if and only if $\tilde{\mu}\leq\mu$ and $\tilde{\nu}\leq\nu$. We set $|(\mu,\nu)|:=|\mu|+|\nu|$ and for each non-negative integer $j$ we write $(\tilde{\mu},\tilde{\nu})<_{j}(\mu,\nu)$ if $(\tilde{\mu},\tilde{\nu})\leq(\mu,\nu)$ and $|(\tilde{\mu},\tilde{\nu})|+j=|(\mu,\nu)|$. When $j=1$, we sometimes write~$\lessdot$ instead of~$<_{1}$. In this case either $\tilde{\mu}=\mu$ or $\tilde{\nu}=\nu$. Lastly, we set $F_{(\mu,\nu)}^{(2)}$ to be the number of directed paths in the lattice $\mathbb{Y}\times \mathbb{Y}$ from $(\emptyset,\emptyset)$ to $(\mu,\nu)$ and similarly we denote by $F_{(\mu,\nu)/(\tilde{\mu},\tilde{\nu})}^{(2)}$ the number of paths in $\mathbb{Y}\times \mathbb{Y}$ from $(\tilde{\mu},\tilde{\nu})$ to $(\mu,\nu)$. One then immediately has that
\begin{equation}\label{eq:F2withF1}
	F_{(\mu,\nu)}^{(2)} = \binom{\lvert \mu \rvert + \lvert \nu \rvert}{\lvert \mu \rvert} F_\mu F_\nu,
	\qquad \qquad
	F_{(\mu,\nu)/(\tilde{\mu},\tilde{\nu})}^{(2)} = \binom{\lvert \mu/\tilde{\mu} \rvert + \lvert \nu/\tilde{\nu} \rvert}{\lvert \mu/\tilde{\mu} \rvert} F_{\mu/\tilde{\mu}} F_{\nu/\tilde{\nu}}.
\end{equation}

\subsection{2-quotients and 2-cores}\label{sec:coresandquotients}
The following construction of the 2-core, labelled by $k$, and the 2-quotient is based on~\cite[I.1,~Ex. 8]{MacDonald} and is comparable with~\cite[Definition~4.5]{Clarkson_GomezUllate_Grandati_Milson}. An equivalent construction using the language of an abacus is well-known and is clearly explained, for example, in~\cite{Wildon}. 

Define the map 
\begin{equation}\label{eq:Phi}
	\Phi: \mathbb{Y}\times\mathbb{Y}\times\mathbb{Z} \to \mathbb{Y}: (\mu,\nu,k) \mapsto \lambda
\end{equation} 
in the following way. For a given pair of partitions $(\mu,\nu)$ and integer $k$, pick two non-negative integers $s$ and $s'$ such that 
\begin{equation}\label{eq:shift}
	\ell(\nu)+s'-\ell(\mu)-s = k
\end{equation}
and consider the following equivalent Maya diagrams 
\begin{equation*}
	M^{(0)}=M_{\mu}+s,
	\qquad \qquad
	M^{(1)}=M_{\nu}+s'.
\end{equation*}
Next, define a third Maya diagram $M$ so that the 2-modular decomposition of 
$M,$ see~\cite{Clarkson_GomezUllate_Grandati_Milson}, is given by $(M^{(0)},M^{(1)})$. That is, the elements in $M$ are such that 
\begin{equation}\label{eq:twomodulardecomposition}
	M^{(i)}
		= \{m\in\mathbb{Z} \mid 2m+i\in M \}
\end{equation}
for $i=0,1$. Finally, the image $\Phi(\mu,\nu,k)$ is the unique partition~$\lambda$ such that $M$ is equivalent to~$M_{\lambda}$. We give an example.

\begin{example}\label{ex:quotient}
Let $\mu=(3,2)$, $\nu=(1)$ and $k=1$. We choose  $s=1$ and $s'=3$ so that~\eqref{eq:shift} holds. The upper part in Figure~\ref{fig:quotient} represents the Maya diagrams~$M^{(0)}$~and~$M^{(1)}$. Subsequently, the Maya diagram $M$ is obtained by taking each box of $M^{(0)}$ and $M^{(1)}$ alternatively as indicated in the figure. Finally, observe that $M=M_{\lambda}+2$ for $\lambda=(4^2,2^2,1)$ and therefore ${\Phi((3,2),(1),1)=(4^2,2^2,1)}$.
\end{example}

\begin{figure}[t]
	\centering
	\begin{tikzpicture}[scale=0.5]
	\draw[fill=black!20!white,draw=none] (-2.5,2) rectangle (10.5,1);
	\draw[very thin,color=gray] (-2.5,2) grid (10.5,1);
	\draw[thick,color=black] (0,2.4) -- (0,0.6);
	\draw (-3,1.5) node {$\dots$};
	\draw (11,1.5) node {$\dots$};
	\foreach \x in {-2,-1,0,3,5}
	{
		\draw (\x+0.5,1.5) node {$\bullet$};
	}
	\draw[left] (-4,1.5) node {$M^{(0)}= M_{\mu}+1$};
	\draw[right] (11.5,1.5) node {$M^{(0)}= \{5,3,0,-1,-2,\dots\}$};
	
	\draw[very thin,color=gray] (-2.5,0) grid (10.5,-1);
	\draw[thick,color=black] (0,0.4) -- (0,-1.4);
	\draw (-3,-0.5) node {$\dots$};
	\draw (11,-0.5) node {$\dots$};
	\foreach \x in {-2,-1,0,1,2,4}
	{
		\draw (\x+0.5,-0.5) node {$\bullet$};
	}
	\draw[left] (-4,-0.5) node {$M^{(1)}= M_{\nu}+3$};
	\draw[right] (11.5,-0.5) node {$M^{(1)}= \{4,2,1,0,-1,-2,\dots\}$};
	\end{tikzpicture}
	
	\vspace{0.5cm}
	
	\begin{tikzpicture}[scale=0.5]
	\foreach \x in {0,2,4,6,8,10,12,14}
	{
		\draw[fill=black!20!white,draw=none] (-2+\x,1) rectangle (-1+\x,0);
	}
	\draw[very thin,color=gray] (-2.5,1) grid (13.5,0);
	\draw[thick,color=black] (0,1.4) -- (0,-0.4);
	\draw (-3,0.5) node {$\dots$};
	\draw (14,0.5) node {$\dots$};
	\foreach \x in {-2,-1,0,1,3,5,6,9,10}
	{
		\draw (\x+0.5,0.5) node {$\bullet$};
	}
	\draw[right] (15,0.5) node {$M= \{10,9,6,5,3,1,0,-1,-2,\dots\}$};
	\end{tikzpicture}
	\caption{The Maya diagrams corresponding to Example~\ref{ex:quotient}.}
	\label{fig:quotient}
\end{figure}
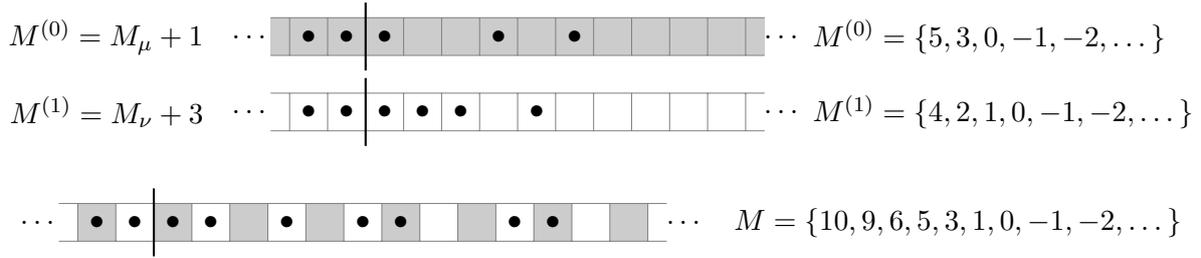

The map~\eqref{eq:Phi} is well-defined because, although we have one degree of freedom in choosing $s$ and $s'$, see~\eqref{eq:shift}, any other choice leads to two Maya diagrams $\widetilde{M}^{(0)}$ and $\widetilde{M}^{(1)}$ that are equivalent to $M^{(0)}$ and $M^{(1)}$. Then $\widetilde{M}$ is equivalent to $M$ and so we end up with the same partition. We trivially have that by construction $\Phi$ is surjective, but not injective. In fact, one can show that
\begin{equation}\label{eq:kand-1-k}
	\Phi(\mu,\nu,k)=\Phi(\tilde{\mu},\tilde{\nu},\tilde{k}) 
	\quad
	\Leftrightarrow
	\quad
	(\mu,\nu,k)=(\tilde{\nu},\tilde{\mu},-\tilde{k}-1).
\end{equation}
Hence the restriction of $\Phi$ to $\mathbb{Y}\times\mathbb{Y}\times\mathbb{Z}_{\geq0}$ is a bijection.

\begin{definition}
For any partition $\lambda$, take the ordered pair $(\mu,\nu)$ and integer $k\geq0$ such that $\lambda=\Phi(\mu,\nu,k)$. Then we call $(\mu,\nu)$ the 2-quotient (shortly the \textbf{quotient}) of $\lambda$. The partition $\bar{\lambda}:=(k,k-1,\dots,2,1)$ is called the 2-core (shortly the \textbf{core}) of $\lambda$.
\end{definition}

The precise ordering of the partitions $\mu$ and $\nu$ in the 2-quotient specified by~\eqref{eq:twomodulardecomposition} is necessary for our purposes since it uniquely distinguishes between $(\mu,\nu)$ and $(\nu,\mu)$ whenever $\mu \ne \nu$. Our ordering matches that which are given for 2-quotients in~\cite{Wildon}. With our ordering convention, the quotient of the conjugate partition $\lambda'$ is $(\nu',\mu')$ and its core is $(k,k-1,\dots,2,1)$ obviously.

Let $\tilde{\lambda}$ be a partition  obtained by removing a domino tile ($2\times1$ or $1\times2$ rectangle) from the Young diagram of~$\lambda$. Removing any  domino corresponds in the Maya diagram picture to swapping a filled box with an empty box two places to its left. Removing a horizontal domino in a partition is equivalent to exchanging an ordered triplet of empty, empty and filled boxes with filled, empty, empty boxes. Similarly, deleting a vertical domino is represented by exchanging a triplet of empty, filled, filled boxes with filled, filled, empty boxes. Therefore we have that the quotient of~$\tilde{\lambda}$, denoted by $(\tilde{\mu},\tilde{\nu})$, must satisfy $(\tilde{\mu},\tilde{\nu})\lessdot(\mu,\nu)$ where $(\mu,\nu)$ is the quotient of~$\lambda$. This then implies that at most~$|\mu|+|\nu|$ dominoes can be removed from~$\lambda$. After deleting this maximal number of dominoes, one ends up with the core $\bar{\lambda}=(k,k-1,\dots,2,1)$. Defining the skew Young diagram $\lambda /\bar{\lambda}$ to be the set-difference between the Young diagram of $\lambda$ and $\bar{\lambda}$, it is therefore trivial that one can always tile the skew Young diagram $\lambda/\bar{\lambda}$ with dominoes. Moreover
\begin{equation}\label{eq:sizelambda}
	|\lambda| = |\bar{\lambda}|+2(|\mu|+|\nu|)
\end{equation}
and, obviously, $|\bar{\lambda}|=k(k+1)/2$.

In accordance with the standard conventions, if $\tilde{\lambda}$ is obtained by removing a domino tile from~$\lambda$, then the height is defined to be  $\htt(\lambda/\tilde{\lambda})=1$ if the domino tile is vertical, and $\htt(\lambda/\tilde{\lambda})=0$ if it is horizontal. Both numbers equal the number of rows that the domino occupies minus one. More generally, if $\tilde{\lambda}$ is obtained by removing a certain number of domino tiles from $\lambda$, of which $m$ tiles are vertical, then we set $\htt_2(\lambda/\tilde{\lambda})=m$, which is the sum of all individual heights. The subscript~2 indicates that we remove domino tiles, that is border strips of size~2. It is well-known that the parity of $\htt_2(\lambda/\tilde{\lambda})$ is independent of the choice of tiling. In other words, the parity of the number of vertical tiles is invariant~\cite{MacDonald}.

If one considers the hook lengths in the Young diagram of a partition $\lambda$, one can determine the (unordered) elements $\mu$ and $\nu$ of the quotient and its core. We illustrate this explicitly in Figure~\ref{fig:hooklengthscore} for our running example partition $\lambda=(4^2,2^2,1)$. Namely, note that in this example, there are 6 cells with an even hook length and 7 with an odd hook length. In general, there will always be at least as many odd hook lengths as even hook lengths. In fact, the difference (in the example $7-6=1$) is a triangular number for every partition, since it is the size of the core of the partition. Next, we shade all the cells with an even hook length, using two colours. If two  such cells are in the same row or column, they are required to have the same colour. It can be proven that this divides the cells with even hook lengths into two (possibly-empty) groups; see~\cite[I.1,~Ex. 8]{MacDonald}. These two groups form the Young diagrams of partitions $\mu$ and $\nu$, as shown on the right in Figure~\ref{fig:hooklengthscore}. Moreover, the hook lengths in the shaded cells are precisely twice the hook lengths in the diagrams of $\mu$ and $\nu$. The order of the partitions is not easily read off from the hook lengths, but for the following formulas, this does not matter. We write $H_{\odd}(\lambda)$ for the product of all odd hook lengths of~$\lambda$, and $H_{\even}(\lambda)$ for the product of all even hook lengths of~$\lambda$. Using~\eqref{eq:Flambdahooklengths}, it is clear that
\begin{equation}\label{eq:Flambdaquotient}
	F_{\lambda} = \frac{|\lambda|!}{H_{\odd}(\lambda) H_{\even}(\lambda)} 
		= \frac{|\lambda|!}{H_{\odd}(\lambda) 2^{|\mu|+|\nu|} H(\mu) H(\nu)}.
\end{equation}
One also observes that the hook lengths of all cores are odd, where a core is said to be a partition that has empty quotient or, equivalently, a partition that is its own core. 

\begin{figure}[t]
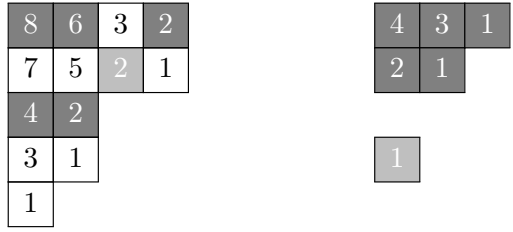

	\centering
	\begin{ytableau}
		*(gray) {\color{white} 8} & *(gray) {\color{white} 6} & 3 & *(gray) {\color{white} 2} \\
		7 & 5 & *(lightgray) {\color{white} 2} & 1 \\
		*(gray) {\color{white} 4} & *(gray) {\color{white} 2} \\
		3 & 1 \\
		1
	\end{ytableau}
	\qquad \qquad \qquad
	\begin{ytableau}
		*(gray) {\color{white} 4} & *(gray) {\color{white} 3} & *(gray) {\color{white} 1} \\
		*(gray) {\color{white} 2} & *(gray) {\color{white} 1} \\
		\none \\
		*(lightgray) {\color{white} 1}
	\end{ytableau}
	\caption{Left: the Young diagram of $(4^2,2^2,1)$ and its hook lengths. Right: the Young diagrams of $\mu=(3,2)$ and $\nu=(1)$, which form the quotient $(\mu,\nu)$, and their hook lengths.}
	\label{fig:hooklengthscore}
\end{figure}

\begin{remark}
The above definitions can easily be generalized to the notion of $p$-quotients and $p$-cores  using the $p$-modular decomposition of a Maya diagram; see Section~\ref{sec:generalization}. This is connected with removing a border strip of size $p$ from a Young diagram, where a border strip is a skew Young diagram that is connected and does not contain any $2\times2$ squares~\cite{MacDonald,Stanley_EC2}. Note that border strips of size 2 are actually dominoes.
\end{remark}

We explicitly state the quotient $(\mu,\nu)$ and core $\bar \lambda$ of some specific partitions, for easy referencing. To this end, let $\lceil \cdot \rceil$ denote the ceiling function and $\lfloor \cdot \rfloor$ represent the floor function. We start with the trivial partitions.

\begin{lemma}\label{lem:CoreAndQuotientTrivialPartition}
Any trivial partition $\lambda=(n)$ has core and quotient given by
\begin{equation*}
	\bar{\lambda}
		=\begin{cases}
			\emptyset & \text{if $n$ even,} \\
			(1) & \text{if $n$ odd,}
		\end{cases}
	\qquad \qquad
	(\mu,\nu)
		=\big(
		\emptyset
		, 
		(\lfloor n/2\rfloor)
		\big).
	\end{equation*}
\end{lemma}

In the context of orthogonality for exceptional Hermite polynomials, one is interested in \textit{even} partitions~\cite{Duran-Hermite,GomezUllate_Grandati_Milson-Hermite,Haese-Hill_Hallnas_Veselov}. 

\begin{lemma}
An even partition $\lambda=(\lambda_1^2,\lambda_2^2,\dots,\lambda_{l}^2)$ has empty core and quotient given by
\begin{equation*}
	(\mu,\nu)
		= \big(
		(\lceil \lambda_1/2 \rceil, \lceil \lambda_2/2 \rceil,\dots, \lceil \lambda_{l}/2 \rceil)
		,
		(\lfloor \lambda_1/2 \rfloor, \lfloor \lambda_2/2 \rfloor,\dots, \lfloor \lambda_{l}/2 \rfloor)
		\big).
\end{equation*}      
\end{lemma}
The generalized Hermite polynomials, which appear in rational solutions of the fourth Painlev\'e equation, are the Wronskian Hermite polynomials associated to partitions whose Young diagram has a rectangular shape~\cite{Buckingham,Clarkson-zeros,Clarkson-PIV,Masoero_Roffelsen,Masoero_Roffelsen-2019,VanAssche}. The core and quotient of such partitions can easily be  deduced.

\begin{lemma}
Any partition $\lambda=(m^n)$ has core and quotient given by
\begin{equation*}
	\bar{\lambda}
		=\begin{cases}
			\emptyset & \text{if $|\lambda|$ even,} \\
			(1) & \text{if $|\lambda|$ odd,}
		\end{cases}
	\qquad \qquad
	(\mu,\nu)
		=\big(
		(\lceil m/2 \rceil^{\lfloor n/2 \rfloor})
		,
		(\lfloor m/2 \rfloor^{\lceil n/2 \rceil})
		\big).
\end{equation*}
\end{lemma}

 
\section{Factorization of Wronskian Hermite polynomials}\label{sec:2coresandWHP}
In this section, we prove that a Wronskian Hermite polynomial can be factorized as in~\eqref{eq:factorizationHe}. The main idea of the proof is to use the generating recurrence relation for Wronskian Hermite polynomials obtained in~\cite[Theorem~3.1]{Bonneux_Stevens}, which expresses $\He_{\lambda}$ in terms of polynomials of lower degree. Here it is convenient to rephrase the recurrence relation in terms of quotient partitions. Using~\cite[Theorem~3.1 and Proposition~3.5]{Bonneux_Stevens}, we have 
\begin{equation}\label{eq:generatingrecurrence}
	F_{\lambda} \He_{\lambda}(x)
		= \frac{x}{|\lambda|} F_{\lambda}\He'_{\lambda}(x) - (|\lambda|-1) \sum_{(\tilde{\mu},\tilde{\nu})\lessdot(\mu,\nu)} (-1)^{\htt_2(\lambda/\tilde{\lambda})} F_{\tilde{\lambda}} \He_{\tilde{\lambda}}(x)
\end{equation}
for any non-empty partition $\lambda$ with quotient $(\mu,\nu)$ and where $\tilde{\lambda}$ has quotient $(\tilde{\mu},\tilde{\nu})$ and the same core as $\lambda$. In this way, the sum in~\eqref{eq:generatingrecurrence} runs precisely over all partitions $\tilde{\lambda}$ that are obtained by removing one domino tile from $\lambda$. Writing the sum as a sum of predecessors in the lattice $\mathbb{Y}\times \mathbb{Y}$ is, however, more convenient for further analysis than writing it as a sum over predecessors of~$\lambda$ in the Young lattice. 

\begin{theorem}\label{thm:WHdecomposition}
For any partition $\lambda$ with core $\bar{\lambda}$ and quotient $(\mu,\nu)$ we have
\begin{equation}\label{eq:WHdecomposition}
	\He_{\lambda}(x)
		=x^{|\bar{\lambda}|} \, R_{\lambda}(x^2)
\end{equation}
where $R_{\lambda}$ is a monic polynomial of degree $|\mu|+|\nu|$ with non-vanishing constant coefficient
\begin{equation}\label{eq:R0}
	R_{\lambda}(0)
		= (-1)^{h_{\lambda}} \frac{H_{\odd}(\lambda)}{H(\bar{\lambda})}
\end{equation}
and $h_{\lambda}=\htt_2(\lambda/\bar{\lambda})+(|\lambda|-|\bar{\lambda}|)/2$.
\end{theorem}

\begin{remark}
Since $|\bar{\lambda}|=k(k+1)/2$, we note that this factorization proves the observation in~\cite{Felder_Hemery_Veselov}. There it is claimed that the multiplicity of the zero at the origin  equals $(p-q)(p-q+1)/2$ where $p$, respectively $q$, denotes the number of odd, respectively even, elements in the degree vector. One easily shows that $\lvert \bar{\lambda} \rvert = (p-q)(p-q+1)/2$, for example, by using induction on the number of domino tiles that are added to the core. Namely, it is straightforward to observe that adding a domino tile to a partition leaves the number $p-q$ invariant, except the case were a horizontal domino is added as a new row. In that case $p-q$ changes to $q-p-1$. Thus if $p\ge q$ then $k=p-q$, otherwise $k=q-p-1$. This is directly related to~\eqref{eq:kand-1-k}.
\end{remark}

\begin{remark}
For any partition $\lambda$ with core $\bar{\lambda}$, the skew diagram $\lambda/\tilde{\lambda}$ can be tiled with dominoes. The number of dominoes equals $(|\lambda|-|\bar{\lambda}|)/2$ and the parity of the number of vertical dominoes is given by $\htt_2(\lambda/\bar{\lambda})$. Therefore, the parity of the quantity $h_\lambda$, defined in Theorem~\ref{thm:WHdecomposition}, equals the parity of the number of horizontal dominoes in a tiling of $\lambda/\tilde{\lambda}$.
\end{remark}

\begin{proof}[Proof of Theorem \ref{thm:WHdecomposition}]
We prove the theorem by induction on $\lvert \mu \rvert + \lvert \nu \rvert$. 

For $|\mu|+|\nu|=0$ we have $\lambda=\bar{\lambda}=(k,k-1,\dots,2,1)$ for some integer $k\geq0$, and so $\He_{\lambda}(x)=x^{|\lambda|}$; see Lemma 4.2 in~\cite{Bonneux_Stevens}. Hence, if we define $R_{\lambda}(x)=1$, then~\eqref{eq:WHdecomposition} holds. Since all hook lengths for the partition $(k,k-1,\dots,2,1)$ are odd,~\eqref{eq:R0} is also satisfied.

Now take $|\mu|+|\nu|>0$ and consider the $m^\textrm{th}$-derivative of both sides of~\eqref{eq:generatingrecurrence} for some integer $m\geq0$. Then, evaluating both sides at zero gives
\begin{equation*}
	F_{\lambda} \He^{(m)}_{\lambda}(0)
		= \frac{m}{|\lambda|} F_{\lambda}\He^{(m)}_{\lambda}(0) - (|\lambda|-1) \sum_{(\tilde{\mu},\tilde{\nu})\lessdot(\mu,\nu)} (-1)^{\htt(\lambda/\tilde{\lambda})} F_{\tilde{\lambda}} \He^{(m)}_{\tilde{\lambda}}(0).
\end{equation*}
Combining terms and subsequently dividing both sides by $m!$ leads to the equality
\begin{equation}\label{eq:WHdecompositionproof1}
	F_{\lambda}\frac{\He^{(m)}_{\lambda}(0)}{m!}
		= -\frac{|\lambda|(|\lambda|-1)}{|\lambda|-m} \sum_{(\tilde{\mu},\tilde{\nu})\lessdot(\mu,\nu)} (-1)^{\htt_2(\lambda/\tilde{\lambda})} F_{\tilde{\lambda}} \frac{\He^{(m)}_{\tilde{\lambda}}(0)}{m!}.
\end{equation}
By applying the induction hypothesis, we conclude that for $0\leq m<\lvert \bar{\lambda} \rvert$, all terms in the sum of~\eqref{eq:WHdecompositionproof1} are zero, since the core of $\tilde{\lambda}$ is also $\bar{\lambda}$. Hence, for all such $m$, we have $\He_\lambda^{(m)}(0)=0$, and so the multiplicity of the zero of $\He_\lambda$ at the origin is at least $\lvert \bar{\lambda} \rvert$. Moreover, it is well-known that the Wronskian Hermite polynomial is an even or odd polynomial depending on its degree; see, for example,~\cite{GomezUllate_Grandati_Milson-Hermite}. Therefore  $\He_\lambda$ can be decomposed as in~\eqref{eq:WHdecomposition}. By evaluating the total degree, we need to have $\lvert \lambda \rvert = \lvert \bar{\lambda} \rvert + 2\deg (R_\lambda)$, and so by~\eqref{eq:sizelambda}, we have $\deg(R_\lambda)=\lvert \mu \rvert + \lvert \nu \rvert$. Since $\He_\lambda$ is monic, $R_\lambda$ is also monic. 

The only thing that is now left to prove is that $R_\lambda(0)$ satisfies~\eqref{eq:R0}. For this, we set $m=\lvert \bar{\lambda} \rvert$ in~\eqref{eq:WHdecompositionproof1} and use the induction hypothesis for $R_{\tilde{\lambda}}(0)$. This then yields
\begin{equation}\label{eq:WHdecompositionproof2}
	F_{\lambda} R_{\lambda}(0)
		= (-1)^{h_\lambda} \frac{|\lambda|(|\lambda|-1)}{|\lambda|-|\bar{\lambda}|}
		\sum_{(\tilde{\mu},\tilde{\nu})\lessdot(\mu,\nu)} F_{\tilde{\lambda}} \frac{H_{\odd}(\tilde{\lambda})}{H(\bar{\lambda})} 
\end{equation} 
because 
\begin{equation*}
	1+\htt_2(\lambda/\tilde{\lambda})+h_{\tilde{\lambda}} \equiv h_{\lambda}  \mod 2.
\end{equation*}
Rewriting~\eqref{eq:WHdecompositionproof2} using~\eqref{eq:sizelambda} and~\eqref{eq:Flambdaquotient} leads to
\begin{equation*}
	R_{\lambda}(0)
		=  (-1)^{h_\lambda} \frac{H_{\odd}(\lambda)}{H(\bar{\lambda})}
	\frac{H(\mu) H(\nu)}{|\mu|+|\nu|} \sum_{(\tilde{\mu},\tilde{\nu})\lessdot(\mu,\nu)} \frac{1}{H({\tilde{\mu}}) H(\tilde{\nu})}.
\end{equation*}
Finally, we see that the right-hand side is non-zero and equal to the expression in~\eqref{eq:WHdecomposition} because
\begin{equation*}
	\sum_{(\tilde{\mu},\tilde{\nu})\lessdot(\mu,\nu)} \frac{1}{H({\tilde{\mu}}) H(\tilde{\nu})}
		= \frac{1}{H(\nu) \lvert \tilde{\mu} \rvert!}\sum_{\tilde{\mu} \lessdot \mu} F_{\tilde{\mu}} + \frac{1}{H(\mu) \lvert \tilde{\nu} \rvert!}\sum_{\tilde{\nu} \lessdot \nu} F_{\tilde{\nu}} 
		= \frac{|\mu|+|\nu|}{H({\mu}) H(\nu)}
\end{equation*}
where we have used~\eqref{eq:Flambdahooklengths} several times. This concludes the proof.
\end{proof}

\begin{corollary}\label{cor:int2}
For any partition $\lambda$ with core $\bar{\lambda}$ we have that $H_{\odd}(\lambda)/H(\bar{\lambda}) \in \mathbb{Z}$.
\end{corollary}
\begin{proof}
This follows directly from Theorem~\ref{thm:WHdecomposition} and the property that $\He_{\lambda}(x)\in\mathbb{Z}[x]$ for any partition $\lambda$, which is proven in~\cite{Bonneux_Hamaker_Stembridge_Stevens}.
\end{proof}

The result of Corollary~\ref{cor:int2} trivially extends to $H(\lambda)/H(\bar{\lambda}) \in \mathbb{Z}$. Moreover, in Corollary~\ref{cor:intp} we give the natural extension to the $p$-core cases for $p>2$. 

\begin{remark}
As mentioned at the end of Section \ref{sec:introduction}, both corollaries also follow from Theorem~4.4 in~\cite{Bessenrodt} as mentioned in Remark~4.5. In \cite{Bessenrodt}, the inclusion of multisets of hooklengths is used, whereas we use the integrality of the coefficients of certain polynomials. Additionally, Theorem~4.4 in~\cite{Bessenrodt} gives an explicit way to calculate the integer valued ratio.
\end{remark}

If $\lambda$ is a core, that is $\lambda=\bar{\lambda}$, we  have $\He_{\lambda}(x)=x^{|\lambda|}$, and so recover the result of Lemma 4.2 in~\cite{Bonneux_Stevens}. For an arbitrary partition $\lambda$, the factorization can be written as
\begin{equation}\label{eq:decomposition}
	\He_{\lambda}(x)
		= \He_{\bar{\lambda}}(x) R_{\lambda}(x^2)
\end{equation}
where $\bar{\lambda}$ denotes the core of $\lambda$. In Section~\ref{sec:coefficients} we establish an explicit formula for the coefficients of~$R_{\lambda}$. As an intermediate result, we present the following corollary.

\begin{corollary}\label{cor:Rlambda}
For any partition $\lambda$ with quotient $(\mu,\nu)$ we have
\begin{equation*}
	R_{\lambda}(x)
		= (-1)^{|\mu|+|\nu|} \, R_{\lambda'}(-x).
\end{equation*}
where $\lambda'$ denotes the conjugated partition of $\lambda$. In particular, if the quotient is of the form $(\mu,\mu')$ for some $\mu$, then $R_{\lambda}$ is an even polynomial.
\end{corollary}
\begin{proof}
It is well-known that $\He_{\lambda}(x) = i^{|\lambda|} \He_{\lambda'}(-ix)$, see for example~\cite{Curbera_Duran}, and therefore the result can  be obtained directly from Theorem~\ref{thm:WHdecomposition}.
\end{proof}

We believe that the converse is also true and we therefore offer the following conjecture. For this, we used the computer software Maple\textsuperscript{{\tiny TM}} to check that the statement indeed holds for all partitions of size at most~35.

\begin{conjecture}
The remainder polynomial $R_\lambda$, as defined in \eqref{eq:decomposition}, is an even polynomial if and only if $\lambda$ is self-conjugate.
\end{conjecture}

If the conjecture is true, then it states that only for self-conjugate partitions we have 
\begin{equation*}
	\He_{\lambda}(x) 
		= x^{|\bar{\lambda}|} \widetilde{R}(x^4)
\end{equation*}
for some polynomial $\widetilde{R}$ with non-zero constant term and where~$\bar{\lambda}$ is the core of~$\lambda$. A subclass of these polynomials is of main interest in~\cite{Carr_Dorey_Dunning}. In that paper, the Wronskian Hermite polynomials associated with all self-conjugate partitions that have empty core are studied. It turns out that this class of polynomials labels the Schr\"odinger equations describing the excited states of the ordinary differential equation / integrable model correspondence for the untwisted, massless sine-Gordon model at its free fermion point. Partition data also plays a key r\^ole in establishing some of the results in \cite{Carr_Dorey_Dunning}. 

\begin{remark}
For Hermite polynomials we have
\begin{align}
	\He_n(0)
		&= (-1)^{n/2} \, (n-1)!! \quad \text{if $n$ even} 
	\label{eq:He0} \\
	\He'_n(0)
		&= (-1)^{(n-1)/2} \, n!! \qquad \vspace{-300em}\text{if $n$ odd}
	\label{eq:He'0}
\end{align}
where for any positive odd integer $n!!=1\cdot3\cdot5 \cdots n$. As $\He_{(n)}(x)=\He_n(x)$ for all $n\geq0$, we now can interpret~\eqref{eq:He0} and~\eqref{eq:He'0} in terms of the odd hook lengths of the trivial partition~$(n)$. This then generalizes to arbitrary partitions as stated in \eqref{eq:R0}.
\end{remark}

 
\section{Coefficients of Wronskian Hermite polynomials}\label{sec:coefficients}
In this section we obtain expressions for all coefficients of the Wronskian Hermite polynomials in terms of partition data.  We consider the factorization of $\He_{\lambda}$ given in~\eqref{eq:WHdecomposition} and write the remainder polynomial as
\begin{equation}\label{eq:Rcoeff}
	R_{\lambda}(x)
		= \sum_{j=0}^{|\mu|+|\nu|} r_{\lambda,j} \, x^{|\mu|+|\nu|-j}
\end{equation}  
where $r_{\lambda,0}=1$ and $r_{\lambda,|\mu|+|\nu|}$ equals the right-hand side of~\eqref{eq:R0}. The recurrence relation~\eqref{eq:generatingrecurrence} for Wronskian Hermite polynomials directly translates to the following recurrence relation for the coefficients $r_{\lambda,j}$ of the remainder polynomial. We use it to prove several interpretations for the coefficients in the subsequent sections.
\begin{proposition}
For any partition $\lambda$ with quotient $(\mu,\nu)$  we have
\begin{equation}\label{eq:recurrencecoefficients}
	F_{\lambda} r_{\lambda,j}
		= -\frac{|\lambda|(|\lambda|-1)}{2j} \sum_{(\tilde{\mu},\tilde{\nu})\lessdot(\mu,\nu)} (-1)^{\htt_2(\lambda/\tilde{\lambda})} F_{\tilde{\lambda}} r_{\tilde{\lambda},j-1}
\end{equation}
for $j=1,2,\dots,\lvert \mu \rvert + \lvert \nu \rvert$ and where $\tilde{\lambda}$ denotes the partition with quotient $(\tilde{\mu},\tilde{\nu})$ and the same core as $\lambda$.
\end{proposition}

We omit the proof since it is an elementary rewriting of~\eqref{eq:generatingrecurrence}, using~\eqref{eq:WHdecomposition} and~\eqref{eq:Rcoeff}, but note that~\eqref{eq:recurrencecoefficients} generates all coefficients if one uses the knowledge that $r_{\lambda,0}=1$ for all $\lambda$.

\subsection{Coefficients in terms of irreducible characters of representations}\label{sec:coeffcharacters}
It is well-known that the Hermite polynomials have the explicit expansion
\begin{equation*}
	\He_n(x)
		=\sum_{j=0}^{\lfloor n/2 \rfloor} (-1)^j \frac{n!}{j! (n-2j)! 2^j} x^{n-2j}
\end{equation*}
for all $n\geq 0$; see for example formula~(5.5.4) in~\cite{Szego}. We offer a generalization of this expansion for Wronskian Hermite polynomials.

\begin{theorem}\label{thm:WHcoeff}
For any partition $\lambda$ we have
\begin{equation}\label{eq:WHcoeff}
	F_\lambda \He_\lambda(x)
		=\sum_{j=0}^{\lfloor\lvert \lambda \rvert/2 \rfloor} (-1)^j \frac{\lvert \lambda \rvert!}{j! (\lvert \lambda \rvert-2j)! 2^j} a(\lambda,j) x^{\lvert \lambda \rvert-2j}
\end{equation}
where $a(\lambda,j)$ is the character of the conjugacy class of the cycle type $(2^j,1^{\lvert \lambda \rvert-2j})$ of the irreducible representation associated to the partition $\lambda$ of the symmetric group $S_{|\lambda|}$. 
\end{theorem}
 
Before explaining the proof of this theorem, we give a very brief introduction to the character theory of the symmetric group, since character values appear in~\eqref{eq:WHcoeff}. The symmetric group~$S_n$ is the group of permutations on $n$ elements. Its representation theory is classical and it is well-known that the \textit{irreducible representations} of $S_n$ are labelled by the partitions $\lambda$ of size $\lvert \lambda \rvert = n$. Much (if not all essential) information of such  an irreducible representation is contained in its \textit{character}, which is a function $\chi_\lambda: S_n \rightarrow \mathbb{Z}$ that has the property of only depending on conjugacy classes: if $\sigma'= \rho \sigma \rho^{-1}$ for $\sigma,\sigma',\rho \in S_n$, then $\chi_\lambda(\sigma)=\chi_\lambda(\sigma')$. It is an easy exercise to show that $\sigma$ and $\sigma'$ are conjugate if and only if they have the same \textit{cycle type}. For example, in $S_5$, we have that $(1 \ 2)(3 \ 4 \ 5)$ and $(1 \ 5)(2 \ 3 \ 4)$ are conjugate, since both have one cycle of length 3 and one cycle of length 2. Ordering the cycle lengths in descending order, one sees that the conjugacy classes are also labelled by the partitions of size $n$. We remark that it is non-trivial that $\chi_\lambda$  takes values in $\mathbb{Z}$, but since it is well-known that this is the case, we omit further details. For a more extensive survey of the representation theory of the symmetric group, we refer to~\cite[I.7]{MacDonald}.

For Theorem~\ref{thm:WHcoeff} we only need the conjugacy classes of cycle type $(2^j, 1^{n-2j})$. This is due to the specific properties of Hermite polynomials. In Section~\ref{sec:generalization}, we show how to generalize the results to obtain the character values for some other cycle types. 

\begin{proof}[Proof of Theorem~\ref{thm:WHcoeff}]
If one applies~\eqref{eq:recurrencecoefficients} recursively, then using $r_{\lambda,0}=1$ one obtains after $j$ iterations that
\begin{equation*}
	F_{\lambda} r_{\lambda,j}
		= (-1)^{j}\frac{|\lambda|!}{(|\lambda|-2j)! j! 2^j} \sum_{(\tilde{\mu},\tilde{\nu})<_{j}(\mu,\nu)} (-1)^{\htt_2(\lambda/\tilde{\lambda})} F_{\tilde{\lambda}}.
\end{equation*}
The last sum is actually equal to the character value $a(\lambda,j)=\chi_\lambda(2^j,1^{\lvert \lambda \rvert -2j})$; see~\cite[I.7,~Ex.~5]{MacDonald} or~\cite[Eq.~(7.75)]{Stanley_EC2}. This immediately concludes the proof.
\end{proof}

\begin{remark}\label{rem:character}
In~\cite{Bonneux_Stevens}, it was shown that the set of polynomials $\He_{\lambda}$ of a fixed degree $n$ satisfy the weighted average property
\begin{equation}\label{eq:average}
	\sum_{\lambda\vdash n} \frac{F_{\lambda}^2}{n!} \He_{\lambda}(x)
		=x^n
\end{equation}
where the sum runs over all partitions of size $n$ and the Plancherel weight is used. Since each polynomial is monic, it follows that the leading term of the average polynomial should be $x^n$. The fact that all other terms vanish can now be interpreted using Theorem~\ref{thm:WHcoeff} 
in terms of the orthogonality of characters. Namely, if we fix $n\geq0$ and invoke~\eqref{eq:WHcoeff} in~\eqref{eq:average}, we obtain
\begin{equation*}
	x^n 
		= \sum_{\lambda \vdash n} \frac{F_\lambda}{n!} \sum_{j=0}^{\lfloor n/2\rfloor} (-1)^j \frac{n!}{j! (n -2j)! 2^j} a(\lambda,j) x^{n -2j}.
\end{equation*}
Matching coefficients and realizing that $F_\lambda = a(\lambda,0)$ is the character value of the identity, for each $0\leq j \leq \lfloor n /2\rfloor$ we have that
\begin{equation*}
	\delta_{j,0} 
		= \frac{1}{n!}\sum_{\lambda \vdash n} F_\lambda a(\lambda,j)
		= \frac{1}{n!} \sum_{\lambda \vdash n} a(\lambda,0) a(\lambda,j)
\end{equation*}
where $\delta_{j,0}=1$ if $j=0$ and 0 otherwise. So the average property \eqref{eq:average} is equivalent to the well-known orthogonality relation of the given characters. 
\end{remark}

\subsection{Coefficients in terms of hook lengths}\label{sec:coeffhooklengths}
We now present explicit formulae for all coefficients $r_{\lambda,j}$ using  the number of directed paths~$F^{(2)}_{(\mu,\nu)}$ in the lattice $\mathbb{Y}\times\mathbb{Y}$ as defined in~\eqref{eq:F2withF1}.

\begin{theorem}\label{thm:coefficients}
Let $\lambda$ be a partition with quotient $(\mu,\nu)$. Then the coefficients of the remainder polynomial~$R_\lambda$, defined in~\eqref{eq:WHdecomposition} and~\eqref{eq:Rcoeff}, are given by
\begin{equation}\label{eq:coefficients}
	r_{\lambda,j}
	= (-1)^j \binom{|\mu|+|\nu|}{j} 
	\sum_{(\tilde{\mu},\tilde{\nu})<_j(\mu,\nu)} (-1)^{\htt_2(\lambda/\tilde{\lambda})}
	\frac{F^{(2)}_{\tilde{\mu},\tilde{\nu}} \,F^{(2)}_{(\mu,\nu)/(\tilde{\mu},\tilde{\nu})}}{F^{(2)}_{\mu,\nu}} 
	\,
	\frac{H_{\odd}(\lambda)}{H_{\odd}(\tilde{\lambda})}
\end{equation}
for $j=0,1,\dots,|\mu|+|\nu|$, where the partition $\tilde{\lambda}$ has
 quotient $(\tilde{\mu},\tilde{\nu})$ and  the same core as $\lambda$.
\end{theorem}
\begin{proof}
We prove this result by induction on $j$. When $j=0$, both sides of~\eqref{eq:coefficients} are trivially one for all partitions $\lambda$. Therefore, we take $j>0$ and consider~\eqref{eq:recurrencecoefficients}. We 
apply the induction hypothesis to the right-hand side of~\eqref{eq:recurrencecoefficients} to obtain
\begin{multline}\label{eq:proofcoefficients0}
	r_{\lambda,j}
	= (-1)^j \frac{|\lambda|(|\lambda|-1)}{2j}  \binom{|\mu|+|\nu|-1}{j-1} 
	\\
	\times
	\sum_{(\tilde{\mu},\tilde{\nu})\lessdot(\mu,\nu)} (-1)^{\htt_2(\lambda/\tilde{\lambda})} \frac{F_{\tilde{\lambda}}}{F_{\lambda}}	
	\sum_{(\hat{\mu},\hat{\nu})<_{j-1}(\tilde{\mu},\tilde{\nu})} (-1)^{\htt_2(\tilde{\lambda}/\hat{\lambda})}
	\frac{F^{(2)}_{\hat{\mu},\hat{\nu}} F^{(2)}_{\tilde{\mu}/\hat{\mu},\tilde{\nu}/\hat{\nu}}}{F^{(2)}_{\tilde{\mu},\tilde{\nu}}} 
	\,
	\frac{H_{\odd}(\tilde{\lambda})}{H_{\odd}(\hat{\lambda})}
\end{multline}
where $\tilde{\lambda}$ has quotient $(\tilde{\mu},\tilde{\nu})$, $\hat{\lambda}$ has quotient $(\hat{\mu},\hat{\nu})$ and the core of both $\tilde{\lambda}$ and $\hat{\lambda}$ are trivially $\bar{\lambda}$. Next, we expand $F_{\lambda}$, $F_{\bar \lambda}$ and $F^{(2)}_{\tilde{\mu},\tilde{\nu}}$, $F^{(2)}_{\mu,\nu}$ using~\eqref{eq:Flambdahooklengths},~\eqref{eq:F2withF1} and~\eqref{eq:Flambdaquotient} to find
\begin{align*}
	&F_{\lambda}
		= \frac{|\lambda|!}{2^{|\mu|+|\nu|}H(\mu)H(\nu) H_{\odd}(\lambda)},
	&&F_{\tilde{\lambda}}
	= \frac{(|\lambda|-2)!}{2^{|\mu|+|\nu|-1}H(\tilde{\mu})H(\tilde{\nu}) H_{\odd}(\tilde{\lambda})},
	\\
	&F^{(2)}_{\tilde{\mu},\tilde{\nu}}
	= \binom{|\tilde{\mu}|+|\tilde{\nu}|}{|\tilde{\mu}|} \frac{|\tilde{\mu}|!}{H(\tilde{\mu})} \frac{|\tilde{\nu}|!}{H(\tilde{\nu})},
	&&F^{(2)}_{\mu,\nu}
	=\binom{|\mu|+|\nu|}{|\mu|}\frac{|\mu|!}{H(\mu)} \frac{|\nu|!}{H(\nu)}.
\end{align*}
Using these expressions and interchanging the sums in~\eqref{eq:proofcoefficients0}, we find the expression for $r_{\lambda,j}$ simplifies to
\begin{equation}\label{eq:proofcoefficients}
	r_{\lambda,j}
		= (-1)^j \binom{|\mu|+|\nu|}{j} 
		\sum_{(\hat{\mu},\hat{\nu})<_{j-1}(\tilde{\mu},\tilde{\nu})}
		(-1)^{\htt_2(\lambda/\hat{\lambda})}
		\frac{F^{(2)}_{\hat{\mu},\hat{\nu}}}{F^{(2)}_{\mu,\nu}}
		\frac{H_{\odd}(\lambda)}{H_{\odd}(\hat{\lambda})}	
		\sum_{(\tilde{\mu},\tilde{\nu})\lessdot(\mu,\nu)}	
		F^{(2)}_{\tilde{\mu}/\hat{\mu},\tilde{\nu}/\hat{\nu}}.
\end{equation}
Finally, we have 
\begin{equation*}
	\sum_{(\tilde{\mu},\tilde{\nu})\lessdot(\mu,\nu)}	
	F^{(2)}_{\tilde{\mu}/\hat{\mu},\tilde{\nu}/\hat{\nu}}
		= 	F^{(2)}_{\mu/\hat{\mu},\nu/\hat{\nu}}
\end{equation*}
and so we conclude that~\eqref{eq:proofcoefficients} leads to~\eqref{eq:coefficients}.
\end{proof}

\begin{remark}\label{rem:weight}
The first fraction in the sum of~\eqref{eq:coefficients} can be seen as a weight:  the denominator is the number of paths from $(\emptyset,\emptyset)$ to $(\mu,\nu)$ in the lattice $\mathbb{Y}\times\mathbb{Y}$, whereas the numerator is the number of such paths 
that pass through $(\tilde{\mu},\tilde{\nu})$. So the sum of all weights for a fixed $j$ is 1.
\end{remark}

\begin{example}\label{ex:coefficientsremainderpolynomial}
In Figure~\ref{fig:dominoprocess} we considered the domino process for the partition $\lambda=(4,2^2,1)$ with core $(2,1)$ and quotient $((2),(1))$. 
Counting the paths shown in Figure~\ref{fig:dominoprocess} from $(4,2^2,1)$ to $(2,1)$, we see that $F_{\mu,\nu}^{(2)}=3$. In this process there are six partitions, whose relevant data is given in Table \ref{tab:coefficientsremainderpolynomial}. Using~\eqref{eq:coefficients} for the coefficients $r_{\lambda,j}$ yields
\begin{align*}
	r_{\lambda,0} 
		&= \binom{3}{0} \frac{3\cdot 1}{3} \frac{105}{105} = 1 \\
	r_{\lambda,1}
		&= -\binom{3}{1} \left(\frac{1\cdot 1}{3} \frac{105}{45} - \frac{2\cdot 1}{3} \frac{105}{63}\right) = -3 \left(\frac{7}{9} -\frac{10}{9}\right)=1 \\
	r_{\lambda,2} 
		&= \binom{3}{2} \left(-\frac{1\cdot 2}{3} \frac{105}{15} +\frac{1\cdot 1}{3}\frac{105}{15}\right) = 3 \left(-\frac{14}{3}+\frac{7}{3}\right) = -7 \\
	r_{\lambda,3} 
		&= -\binom{3}{3} \frac{1 \cdot 3}{3} \frac{105}{3} = -35
\end{align*}
and consequently $\He_{(4,2^2,1)}(x)=x^3(x^6+x^4-7x^2-35)$, which was already stated in Figure~\ref{fig:dominoprocess}.
\end{example}

\begin{table}[t]
	\centering
	\begin{tabular}{| c | c | c | c | c | c | c |}
		\hline
		& & & & & &\\[-1em]
		j & $\tilde{\lambda}$ & $(\tilde{\mu},\tilde{\nu})$ & $F_{\tilde{\mu},\tilde{\nu}}^{(2)}$ & $F^{(2)}_{(\mu,\nu)/(\tilde{\mu},\tilde{\nu})}$ & $H_{\odd}(\tilde{\lambda})$ & $(-1)^{\htt_2(\lambda/\tilde{\lambda})}$ \\
		& & & & & &\\[-1em] 
		\hline
		& & & & & &\\[-1em]
		0 & $(4,2^2,1)$ & ((2),(1)) & 3 & 1 & 105 & $+1$\\
		& & & & & &\\[-1em]
		\hline
		& & & & & &\\[-1em]
		1 & $(2^3,1)$ & $((2),\emptyset)$ & 1 & 1 & 45 & $+1$ \\
		1 & $(4,1^3)$ & ((1),(1)) & 2 & 1 & 63 & $-1$ \\
		& & & & & &\\[-1em]
		\hline
		& & & & & &\\[-1em]
		2 & $(2,1^3)$ & $((1),\emptyset)$ & 1 & 2 & 15 & $-1$ \\
		2 & (4,1) & $(\emptyset,(1))$ & 1 & 1 & 15 & $+1$\\
		& & & & & &\\[-1em]
		\hline
		& & & & & &\\[-1em]
		3 & (2,1) & $(\emptyset,\emptyset)$ & 1 & 3 & 3 & $+1$ \\
		\hline 
	\end{tabular}
	\caption{Partition data related to Example \ref{ex:coefficientsremainderpolynomial}.}
	\label{tab:coefficientsremainderpolynomial}
\end{table}

For the results in Section~\ref{subsec:coeffpoly}, it is convenient to rewrite~\eqref{eq:coefficients} using~\eqref{eq:R0}. One immediately obtains that the constant terms of the polynomials $(R_\lambda)_{\lambda\in\mathbb{Y}}$ carry all information about all other terms in the polynomials.

\begin{corollary}\label{cor:coefficientsR0}
Let $\lambda$ be a partition with quotient $(\mu,\nu)$. Then the coefficients of the polynomial~$R_\lambda$, defined in~\eqref{eq:WHdecomposition} and~\eqref{eq:Rcoeff}, are given by
\begin{equation}\label{eq:coefficientsR0}
	r_{\lambda,j} 
		= \binom{\lvert \mu \rvert + \lvert \nu \rvert}{j} \sum_{(\tilde{\mu},\tilde{\nu})<_j(\mu,\nu)} \frac{F^{(2)}_{\tilde{\mu},\tilde{\nu}} \,F^{(2)}_{(\mu,\nu)/(\tilde{\mu},\tilde{\nu})}}{F^{(2)}_{\mu,\nu}} \frac{R_\lambda(0)}{R_{\tilde{\lambda}}(0)}
\end{equation}
for $j=0,1,\dots,|\mu|+|\nu|$  where $\tilde{\lambda}$ has quotient $(\tilde{\mu},\tilde{\nu})$ and the same core as $\lambda$.
\end{corollary}

\begin{remark}
It is also possible to rewrite~\eqref{eq:coefficients} using~\eqref{eq:F2withF1} and~\eqref{eq:Flambdaquotient}. This yields
\begin{equation*}
	r_{\lambda,j}
		= (-2)^{-j}
		\sum_{(\tilde{\mu},\tilde{\nu})<_{j}(\mu,\nu)}
		(-1)^{\htt_2(\lambda/\tilde{\lambda})}
		\frac{F_{\mu/\tilde{\mu}} \, F_{\nu/\tilde{\nu}}}{|\mu/\tilde{\mu}|! \, |\nu/\tilde{\nu}|!}
		\frac{H(\lambda)}{H(\tilde{\lambda})}
\end{equation*}
for each $j$. Though this formula is less susceptible to interpretation because the first fraction in the sum cannot be seen as a weight (cf. Remark~\ref{rem:weight}), it uses information about all the hook lengths and not just the odd ones.
\end{remark}

\subsection{Coefficients as polynomials in the length of the core}\label{subsec:coeffpoly}
For a partition $\lambda$ with quotient $(\mu,\nu)$ we have that the polynomial $R_\lambda$ has degree~${\lvert\mu\rvert+\lvert\nu\rvert}$. Therefore, one  naturally asks how the polynomial $R_\lambda$ changes if one fixes the quotient $(\mu,\nu)$ and varies the size $k$ of the core. 

\begin{theorem}\label{thm:coeffpoly}
Fix a pair of partitions $(\mu,\nu)$ and for all $k\geq 0$ let $\Phi(\mu,\nu,k)$ denote the partition with quotient $(\mu,\nu)$ and core $(k,k-1,\dots,2,1)$ as described in Section~\ref{sec:coresandquotients}. Then we have that for all $0\leq j \leq \lvert \mu \rvert + \lvert \nu \rvert$, the coefficient $r_{\Phi(\mu,\nu,k),j}$ of $R_{\Phi(\mu,\nu,k)}$, defined in~\eqref{eq:WHdecomposition} and~\eqref{eq:Rcoeff}, is a polynomial in $k$ of at most degree $j$.
\end{theorem}

\begin{example}
Let $(\mu,\nu)=((4,1),(3))$. Then
\begin{align*}
	R_{\Phi(\mu,\nu,k)}(x)
	=&\,x^8
	+2(2k-15) \, x^7
	-2(2k-5)(2k+33) \, x^6
	+(-48k^3+24k^2+1404k-1230) \, x^5
	\\
	&+120k(2k-7)(2k+9) \, x^4
	+6(2k-7)(2k-5)(2k+1)(2k+7)(2k+9) \, x^3
	\\
	&+2(2k-7)(2k-5)(2k-3)(2k+1)(2k+5)(2k+9) \, x^2
	\\
	&-2(2k-7)(2k-5)(2k-3)(2k+1)(2k+3)(2k+5)(2k+9) \, x
	\\
	&-(2k-7)(2k-5)(2k-3)(2k+1)^2 (2k+3)(2k+5)(2k+9).
\end{align*}
See Table \ref{tab:exampleWH} for the corresponding Wronskian Hermite polynomials for $k=0,1,2,3,4$. 
\end{example}

\begin{table}[t]
	\centering
	\begin{scriptsize}
	\begin{tabular}{| c | c | l |}
		\hline
		& & \\[-0.2cm]
		$k$ & $\Phi((4,1),(3),k) $ & $\He_{\Phi((4,1),(3),k)}(x) $ \\
		\hline
		& & \\[-0.2cm]
		0 & $(7^2,1^2)$ & $x^0(x^{16}-30x^{14}+330x^{12}-1230x^{10}+13230x^6-9450x^4+28350x^2+14175)$\\[0.1cm]
		1 & $(7^2,1^3)$ & $x^1(x^{16}-26x^{14}+210x^{12}+150x^{10}-6600x^8+26730x^6-6930x^4+34650x^2+51975)$\\[0.1cm]
		2 & $(8,6,2,1^3)$ & $x^3(x^{16}-22x^{14}+74x^{12}+1290x^{10}-9360x^{8}+12870x^6+3510x^4-24570x^2-61425)$\\[0.1cm]
		3 & $(9,5,3,2,1^3)$ & $x^6(x^{16}-18x^{14}-78x^{12}+1902x^{10}-5400x^{8}-8190x^6-6930x^4+62370x^2+218295)$\\[0.1cm]
		4 & $(10,4^2,3,2,1^3)$ & $x^{10}(x^{16}-14x^{14}-246x^{12}+1698x^{10}+8160x^{8}+41310x^{6}+59670x^{4}-656370x^{2}-2953665)$\\[0.1cm]
		\hline
	\end{tabular}
	\end{scriptsize}
	\caption{The polynomials $\He_{\Phi((4,1),(3),k)}(x)$ for $k=0,1,2,3,4$. The coefficients are polynomial in~$k$, see Theorem~\ref{thm:coeffpoly}.}
	\label{tab:exampleWH}
\end{table}

\begin{remark}
We do not offer explicit expressions for all the coefficients in terms of $k$; the proof of Theorem~\ref{thm:coeffpoly} is existential instead of constructive. However, we can deduce enough properties of the polynomials to obtain the asymptotic result presented in Section~\ref{sec:asymptotics}.
\end{remark}

The rest of this section is dedicated to the rather technical proof of Theorem~\ref{thm:coeffpoly}. It relies on Corollary~\ref{cor:coefficientsR0}, which expresses the coefficients of the remainder polynomials in terms of their constants. In Proposition~\ref{prop:psi} we give explicit expressions for these constants as polynomials in $k$. First, we give  a useful way of describing the odd hooks of a partition in terms of the two Maya diagrams that represent the quotient. 

\begin{lemma}\label{lem:formulahooklengths0}
Take a pair of partitions $(\mu,\nu)$ and denote their Maya diagrams by $M_{\mu}$ and $M_{\nu}$. Let $s,s'\geq 0$ be integers such that $\ell(\mu)+s=\ell(\nu)+s'$. Then we have
\begin{multline}\label{eq:formulahooklengths0}
H_{\odd}(\Phi(\mu,\nu,0))=\left[\prod_{m\in M_\mu +s} \prod_{\substack{n<m \\ n \not\in M_\nu+s'}} \big( 2(m-n) -1 \big) \right] \\ \times \left[\prod_{n\in M_\nu +s'} \prod_{\substack{m\leq n \\ m \not\in M_\mu+s}} \big( 2(n-m) +1 \big) \right].
\end{multline}
\end{lemma}

This formula is not explicitly in \cite{MacDonald}, but follows quite directly from the concepts therein. For the readers' convenience, we include the idea of the proof. First, note that the hook lengths of the boxes in a Young diagram can be indicated using the corresponding Maya diagram. Namely, for any fixed filled box in a Maya diagram, the set of distances between that box and all empty boxes to the left indicate the hook lengths of a row in the Young diagram, as indicated for an example in the middle image in Figure~\ref{fig:hooklengthsMayadiagram}. We note that it does not matter if we choose the canonical Maya diagram or an equivalent one, since this choice does not alter the relative distances.

Transferring this  interpretation to the Maya diagrams of the quotient, we see that the hook lengths naturally split into two classes. The even hook lengths are given by  twice the distances between an empty and a filled box within the same Maya diagram, while the odd hook lengths are in terms of a filled box in one Maya diagram and an empty box in the other one. This is visualized in the right image in Figure~\ref{fig:hooklengthsMayadiagram}. Writing down this interpretation directly yields the formula~\eqref{eq:formulahooklengths0}, which consists of two finite double products.

\begin{figure}[h]
	\centering
	\begin{tikzpicture}[scale=0.4]
	\footnotesize{\foreach \x in {0,2,4,6,8,10,12}
		{
			\draw[fill=black!20!white,draw=none] (-2+\x,1) rectangle (-1+\x,0);
		}
	\draw[very thin,color=gray] (-2.5,1) grid (11.5,0);
	\draw[thick,color=black] (0,1.4) -- (0,-0.4);
	\draw (-3,0.5) node {$\dots$};
	\draw (12,0.5) node {$\dots$};
	\foreach \x in {-2,-1,1,3,4,7,8}
	{
		\draw (\x+0.5,0.5) node {$\bullet$};
	}

	\draw (0.5,1) to[out=90,in=90] node[pos=0.5,above]{8} (8.5,1);
	\draw (2.5,1) to[out=90,in=90] node[pos=0.25, below]{6} (8.5,1);
	\draw (5.5,1) to[out=90,in=90] node[pos=0.07, above]{3} (8.5,1);
	\draw (6.5,1) to[out=90,in=90] node[pos=0.15,xshift=0.4cm]{2} (8.5,1);
	
	\draw[dashed] (0.5,0) to[out=-90,in=-90] node[pos=0.25,below]{7} (7.5,0);
	\draw[dashed] (2.5,0) to[out=-90,in=-90] node[pos=0.4, yshift=0.15cm]{5} (7.5,0);
	\draw[dashed] (5.5,0) to[out=-90,in=-90] node[pos=0.25,left]{2} (7.5,0);
	\draw[dashed] (6.5,0) to[out=-90,in=-90] node[pos=0.5,xshift=0.3cm,yshift=-0.02cm]{1} (7.5,0);
	
	
	\draw[fill=black!20!white,draw=none] (15.5,3) rectangle (24.5,2);
	\draw[very thin,color=gray] (15.5,3) grid (24.5,2);
	\draw[thick,color=black] (18,3.4) -- (18,1.6);
	\draw (15,2.5) node {$\dots$};
	\draw (25,2.5) node {$\dots$};
	\foreach \x in {-2,-1,2,4}
	{
		\draw (\x+18.5,2.5) node {$\bullet$};
	}
	
	\draw[very thin,color=gray] (15.5,-1) grid (24.5,-2);
	\draw[thick,color=black] (18,-0.6) -- (18,-2.4);
	\draw (15,-1.5) node {$\dots$};
	\draw (25,-1.5) node {$\dots$};
	\foreach \x in {-2,-1,0,1,3}
	{
		\draw (\x+18.5,-1.5) node {$\bullet$};
	}

	\draw (18.5,3) to[out=90,in=90] node[pos=0.10, above]{8} (22.5,3);
	\draw (19.5,3) to[out=90,in=90] node[pos=0.30, xshift=0.1cm, yshift=-0.13cm]{6} (22.5,3);
	\draw (21.5,3) to[out=90,in=90] node[pos=0.10, xshift=-0.1cm,yshift=0.1cm]{2} (22.5,3);
	\draw (20.5,-1) to[out=90,in=-90] node[pos=0.80, xshift=0.05cm, yshift=-0.15cm]{3} (22.5,2);
	
	\draw[dashed] (20.5,-2) to[out=-90,in=-90] node[pos=0.50, yshift=-0.15cm]{2} (21.5,-2);
	\draw[dashed] (18.5,2) to[out=-90,in=90] node[pos=0.25, below]{7} (21.5,-1);
	\draw[dashed] (19.5,2) to[out=-90,in=90] node[pos=0.20, right]{5} (21.5,-1);
	\draw[dashed] (21.5,2) to[out=-90,in=90] node[pos=0.70, xshift=0.1cm]{1} (21.5,-1);
	}
	
	\draw (-8,0.5) node 
		{
			\begin{ytableau}
				8 & 6 & 3 & 2 \\
				7 & 5 & 2 & 1 \\
				{\color{lightgray} 4} & {\color{lightgray} 2} \\
				{\color{lightgray} 3} & {\color{lightgray} 1} \\
				{\color{lightgray} 1}
			\end{ytableau}
		};
\end{tikzpicture}
\caption{The hook lengths of the first two rows of $(4^2,2^2,1)$ displayed in its Young diagram, in its canonical Maya diagram, and in the pair of Maya diagrams associated to the quotient.}
\label{fig:hooklengthsMayadiagram}
\end{figure}
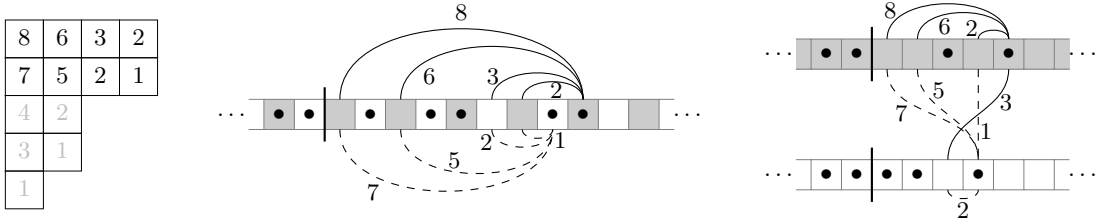 

Motivated by Lemma~\ref{lem:formulahooklengths0} we define the following polynomial.

\begin{definition}
For any pair of partitions $(\mu,\nu)$, with Maya diagrams $M_{\mu}$ and $M_{\nu}$ and integers $s,s'\geq 0$ such that $\ell(\mu)+s=\ell(\nu)+s'$, we set
\begin{multline}\label{eq:defPsi}
\Psi_{\mu,\nu}(z):=(-1)^{h_{\Phi(\mu,\nu,0)}} \left[\prod_{m\in M_\mu +s} \prod_{\substack{n<m \\ n \not\in M_\nu+s'}} \big( 2(m-n) -1-2z \big) \right] \\ \times \left[\prod_{n\in M_\nu +s'} \prod_{\substack{m\leq n \\ m \not\in M_\mu+s}} \big( 2(n-m) +1 +2z \big) \right]
\end{multline}
where $\Phi(\mu,\nu,0)$ represents the partition $\lambda$ with quotient $(\mu,\nu)$ and empty core such that $|\lambda|$ is even, see \eqref{eq:sizelambda}. By definition, we then have ${h_{\Phi(\mu,\nu,0)}=\htt_2(\lambda) +|\lambda|/2}$.
\end{definition}

By Lemma~\ref{lem:formulahooklengths0} and~\eqref{eq:R0} we immediately derive that
\begin{equation}\label{eq:Psi0}
	\Psi_{\mu,\nu}(0)
		= (-1)^{h_{\Phi(\mu,\nu,0)}} H_{\odd}(\Phi(\mu,\nu,0))=R_{\Phi(\mu,\nu,0)}(0).
\end{equation}
Moreover, the main reason for introducing the polynomial $\Psi_{\mu,\nu}(z)$ is that~\eqref{eq:Psi0} can be generalized.

\begin{proposition}\label{prop:psi}
For any pair of partitions $(\mu,\nu)$ and for any $k\geq 0$, we have
\begin{equation}\label{eq:psiR}
\Psi_{\mu,\nu}(k)=(-1)^{h_{\Phi(\mu,\nu,k)}} \frac{H_{\odd}(\Phi(\mu,\nu,k))}{H((k,k-1,\dots,2,1))} = R_{\Phi(\mu,\nu,k)}(0).
\end{equation}
Furthermore, $\Psi_{\mu,\nu}$ is a polynomial in $k$ of degree $\lvert \mu \rvert + \lvert \nu \rvert$, with leading coefficient $(-1)^{|\nu|} \, 2^{|\mu|+|\nu|}$.
\end{proposition}
We note that the last equality in~\eqref{eq:psiR} is precisely~\eqref{eq:R0}. In particular,~\eqref{eq:psiR} proves that $R_{\Phi(\mu,\nu,k)}(0)$ is polynomial in $k$. A short argument, based on the following lemma, now takes us from Proposition~\ref{prop:psi} to Theorem~\ref{thm:coeffpoly}, after which we give the longer argument required to prove Proposition~\ref{prop:psi}. 

\begin{lemma}\label{lem:rationalfunctionintegers}
Suppose that $f$ is a rational function with rational coefficients. If $f(n)$ is an integer for all integers $n\geq 1$, then $f$ is a polynomial.
\end{lemma}
\begin{proof}
Since $f$ is a rational function with rational coefficients, there are two polynomials $p,q\in~\mathbb{Q}[x]$ such that $f=\frac{p}{q}$. Then, there exist polynomials $r,\tilde{p}\in \mathbb{Q}[x]$ such that $f=r+\frac{\tilde{p}}{q}$ and $\deg(\tilde{p})<\deg(q)$. Hence, there is an integer $N$ such that $N \cdot r(x)\in \mathbb{Z}[x]$, whence $N \cdot r(n)$ is an integer for all integers $n\geq 0$. Combining this with the assumption on $f$, this means that $N \cdot \frac{\tilde{p}}{q}$ maps positive integers to integers. We also have that $N \cdot \frac{\tilde{p}}{q}(n) \to 0$ as $n\to\infty$ and so there are infinitely many integers $n$ such that $\tilde{p}(n)=0$. However, $\tilde{p}$ is a polynomial, so it  should be that $\tilde{p}\equiv0$. Hence $f=r$ and therefore $f$ is a polynomial.
\end{proof}

\begin{proof}[Proof of Theorem~\ref{thm:coeffpoly}]
Combining~\eqref{eq:psiR} with~\eqref{eq:coefficientsR0} yields 
\begin{equation*}
	r_{\Phi(\mu,\nu,k),j}=\binom{\lvert \mu \rvert + \lvert \nu \rvert}{j} \sum_{(\tilde{\mu},\tilde{\nu})<_j(\mu,\nu)} \frac{F^{(2)}_{\tilde{\mu},\tilde{\nu}} \,F^{(2)}_{(\mu,\nu)/(\tilde{\mu},\tilde{\nu})}}{F^{(2)}_{\mu,\nu}} \frac{\Psi_{\mu,\nu}(k)}{\Psi_{\tilde{\mu},\tilde{\nu}}(k)}
\end{equation*}
and since all $\Psi_{\mu,\nu}$ are polynomials of degree~$|\mu|+|\nu|$, the above expression is a rational function in $k$ of degree at most $j$. However, we know that for all $k$, the coefficient $r_{\Phi(\mu,\nu,k),j}$ is an integer by~\cite[Corollary~7.1]{Bonneux_Hamaker_Stembridge_Stevens}. By Lemma~\ref{lem:rationalfunctionintegers} we therefore conclude that $r_{\Phi(\mu,\nu,k),j}$ is in fact a polynomial in~$k$ of degree at most~$j$.
\end{proof} 

All that remains to be proven is Proposition~\ref{prop:psi} itself.

\begin{proof}[Proof of Proposition~\ref{prop:psi}]
In part 1 we prove identity~\eqref{eq:psiR}, and then in part 2 we derive the degree and the leading coefficient of $\Psi_{\mu,\nu}$. 
	
\textbf{Part 1.} Take integers $s,s'\geq0$ such that $\ell(\mu)+s=\ell(\nu)+s'$. The odd hooks lengths of the partition with quotient $(\mu,\nu)$ and core 
$k> 0$ may be found 
in the same way as in Lemma~\ref{lem:formulahooklengths0}. In fact  we simply 
need to  replace $s'$ by $s'+k$ to find 
\begin{multline*}
H_{\odd}(\Phi(\mu,\nu,k))=\left[\prod_{m\in M_\mu +s} \prod_{\substack{n<m \\ n \not\in M_\nu+s'+k}} \big( 2(m-n) -1 \big) \right] \\ \times \left[\prod_{n\in M_\nu +s'+k} \prod_{\substack{m\leq n \\ m \not\in M_\mu+s}} \big( 2(n-m) +1 \big) \right]
\end{multline*}
where $\Phi(\mu,\nu,k)$ represents the partition with quotient $(\mu,\nu)$ and core  length $k$. Replacing all~$n$ with $l+k$, we obtain
\begin{multline*}
H_{\odd}(\Phi(\mu,\nu,k))=\left[\prod_{m\in M_\mu +s} \prod_{\substack{l+k<m \\ l \not\in M_\nu+s'}} \big( 2(m-(l+k)) -1 \big) \right] \\ \times \left[\prod_{l\in M_\nu +s'} \prod_{\substack{m\leq l+k \\ m \not\in M_\mu+s}} \big( 2((l+k)-m) +1 \big)\right].
\end{multline*}
Now we recognize that this expression is  similar to $\Psi_{\mu,\nu}(k)$, defined in~\eqref{eq:defPsi}, except for certain factors. Namely,
\begin{equation}\label{eq:prooftechnical0}
H_{\odd}(\Phi(\mu,\nu,k)) = (-1)^{h_{\Phi(\mu,\nu,0)}} \, \Psi_{\mu,\nu}(k) \cdot \frac{\prod_{l\in M_\nu +s'} \prod_{\substack{l < m\leq l+k \\ m \not\in M_\mu+s}} \big( 2((l+k)-m) +1 \big)}{\prod_{m\in M_\mu +s} \prod_{\substack{m-k\leq l<m \\ l \not\in M_\nu+s'}} \big( 2(m-(l+k)) -1 \big)}.
\end{equation}
Now note that all the factors in the numerator are positive, whereas all factors in the denominator are negative. Multiplying the latter factors by $-1$ then yields
\begin{equation}\label{eq:prooftechnical1}
H_{\odd}(\Phi(\mu,\nu,k)) = (-1)^{\beta(\mu,\nu,k)} \, \Psi_{\mu,\nu}(k) \cdot \frac{\prod_{l\in M_\nu +s'} \prod_{\substack{l < m\leq l+k \\ m \not\in M_\mu+s}} \big( 2((l+k)-m) +1 \big)}{\prod_{m\in M_\mu +s} \prod_{\substack{m-k\leq l<m \\ l \not\in M_\nu+s'}} \big( 2((l+k)-m) +1 \big)}
\end{equation}
where
\begin{equation}\label{eq:defalpha}
\beta(\mu,\nu,k):=h_{\Phi(\mu,\nu,0)}+\# \left\{(m,l)\in \mathbb{Z}^2 \mid m\in M_\mu+s, \, l\not\in M_\nu+s', \, m-k \leq l <m\right\}.
\end{equation}
We now make the following two claims.
\begin{description}
	\item[Claim 1:] the parity of $\beta(\mu,\nu,k)$ and $h_{\Phi(\mu,\nu,k)}$ are the same.
	\item[Claim 2:] we have
	\begin{equation*}
	\frac{\prod_{l\in M_\nu +s'} \prod_{\substack{l < m\leq l+k \\ m \not\in M_\mu+s}} \big( 2((l+k)-m) +1 \big)}{\prod_{m\in M_\mu +s} \prod_{\substack{m-k\leq l<m \\ l \not\in M_\nu+s'}} \big( 2((l+k)-m) +1 \big)} = H((k,k-1,\dots,2,1)).
	\end{equation*}
\end{description}
Plugging these two claims into~\eqref{eq:prooftechnical1} yields the following alternative writing of~\eqref{eq:psiR}: 
\begin{equation*}
	H_{\odd}(\Psi(\mu,\nu,k)) = (-1)^{h_{\Phi(\mu,\nu,0)}} \, \Psi_{\mu,\nu}(k) H((k,k-1,\dots,2,1)).
\end{equation*}
We therefore only have to prove both claims to establish part 1.

\noindent \textbf{Proof of Claim 1.}
We use induction on $\lvert (\mu,\nu) \rvert$. If $\lvert (\mu,\nu) \rvert=0$, we have that $(\mu,\nu)=(\emptyset,\emptyset)$, so $\Phi(\mu,\nu,k)=(k,k-1,\dots,2,1)$ for all $k\geq0$. Hence $h_{\Phi(\mu,\nu,k)}=0$ for all $k$. One also  sees that
\begin{equation*}
	\# \left\{(m,l)\in \mathbb{Z}_{\geq 0}^2 \mid m\in M_\emptyset+s, \, l\not\in M_\emptyset+s', \, m-k \leq l <m\right\}=0
\end{equation*}
because $s'-s=k$. This establishes the induction basis.

Next take $\lvert (\mu,\nu) \rvert>0$ and fix $k\geq0$. Then there either exists a partition $\tilde{\mu}$ such that $(\tilde{\mu},\nu)\lessdot (\mu,\nu)$, or there exists a partition $\tilde{\nu}$ such that $(\mu,\tilde{\nu})\lessdot (\mu,\nu)$. As both situations can be proven similarly, we only give the proof in the first situation. In this case, $\mu$ is obtained from $\tilde{\mu}$ by moving a dot in the Maya diagram from some position $t$ to $t+1$. We then see that $\Phi(\mu,\nu,k)$ is obtained from $\Phi(\tilde{\mu},\nu,k)$ by adding a horizontal domino tile if and only if $t-k \not\in M_\nu + s'$. In other words, we have
\begin{equation}\label{eq:claim1proof1}
	h_{\Phi(\mu,\nu,k)} \equiv h_{\Phi(\tilde{\mu},\nu,k)}+ 1 \mod 2 \quad \Leftrightarrow \quad t-k\not\in M_\nu+s'
\end{equation}
for all $k$. Subsequently, define the numbers  
\begin{align*}
	&N_{(\tilde{\mu},\nu)}
		= \# \left\{(m,l)\in \mathbb{Z}^2 \mid m\in M_{\tilde{\mu}}+s, \, l\not\in M_\nu+s', \, m-k \leq l<m\right\} \\
	&N_{(\mu,\nu)}
		= \# \left\{(m,l)\in \mathbb{Z}^2 \mid m\in M_{\mu}+s, \, l\not\in M_\nu+s', \, m-k \leq l <m\right\}
\end{align*}
and observe that
\begin{equation}\label{eq:claim1proof2}
	N_{(\mu,\nu)} \equiv N_{(\tilde{\mu},\nu)} + 1 \mod 2 \quad \Leftrightarrow \quad t-k\not\in M_\nu+s' \text{ or } t \not\in M_\nu+s' \text{ but not both.}
\end{equation}
Next, reconsider \eqref{eq:defalpha} such that combining \eqref{eq:claim1proof1} for $k=0$ and \eqref{eq:claim1proof2} leads to 
\begin{equation}\label{eq:claim1proof3}
	\beta(\mu,\nu,k) \equiv \beta(\tilde{\mu},\nu,k)+ 1 \mod 2 \quad \Leftrightarrow \quad t-k\not\in M_\nu+s'.
\end{equation}
Hence, using the induction hypothesis that the parity of $h_{\Phi(\tilde{\mu},\nu,k)}$ and $\beta(\tilde{\mu},\nu,k)$ are the same, we see that $h_{\Phi(\mu,\nu,k)}$ and $\beta(\mu,\nu,k)$ have the same parity because of \eqref{eq:claim1proof1} and \eqref{eq:claim1proof3}. This proves Claim 1. $\blacktriangleright$

\noindent \textbf{Proof of Claim 2.}
We first need the following observation about counting factors. The number of boxes in the Young diagram of  $\Phi(\mu,\nu,k)$ is equal to $k(k+1)/2+2(\lvert \mu \rvert + \lvert \nu \rvert)$. There are $\lvert \mu \rvert + \lvert \nu \rvert$ boxes with an even hook length, while the other hook lengths are odd. 
By definition, $\Psi_{\mu,\nu}$ consists of $\lvert \mu \rvert + \lvert \nu \rvert$ factors because it only runs over the odd hook lengths of a partition with empty core. The expansion $H_{\odd}(\Phi(\mu,\nu,k))$ has $k(k+1)/2+ \lvert \mu \rvert + \lvert \nu \rvert$ factors. Hence, combining both results, the fraction
\begin{equation*}
\frac{\prod_{l\in M_\nu +s'} \prod_{\substack{l < m\leq l+k \\ m \not\in M_\mu+s}} \big( 2((l+k)-m) +1 \big)}{\prod_{m\in M_\mu +s} \prod_{\substack{m-k\leq l<m \\ l \not\in M_\nu+s'}} \big( 2((l+k)-m) +1 \big)}
\end{equation*}
must have $k(k+1)/2$ factors more in the numerator than in the denominator 
so that  the number of factors in \eqref{eq:prooftechnical0} is preserved. If we now interchange the products in both the numerator and denominator, we obtain
\begin{equation}\label{eq:afterinterchanging}
	\frac{\prod_{i=1}^k \prod_{\substack{l \in M_\nu +s' \\ l+i \not\in M_\mu+s}} \big( 2(k-i)+1\big)}{\prod_{i=1}^k \prod_{\substack{m \in M_\mu +s \\ m-i \not\in M_\nu+s'}} \big( 2(k-i)+1\big)}.
\end{equation}
The counting of factors given above  implies that for any $k$, we have that
\[\sum_{i=1}^k \left(\#\{  l \in M_\nu+s' \mid l+i \not\in M_\mu+s\} - \#\{  m \in M_\mu+s \mid m-i \not\in M_\nu+s'\}\right) = \frac{k(k+1)}{2}.\]
Since all terms in this sum are independent of $k$ and the above identity holds for all $k$, we conclude that
\[\#\{  l \in M_\nu+s' \mid l+i \not\in M_\mu+s\} - \#\{  m \in M_\mu+s \mid m-i \not\in M_\nu+s'\}=i\]
for all $i=1,2\dots,k$. Therefore, \eqref{eq:afterinterchanging} is equal to
\[\prod_{i=1}^k (2(k-i)+1)^i =  H((k,k-1,\dots,2,1))\]
and we have proven Claim 2.$\blacktriangleright$

\textbf{Part 2.} By construction, the degree of $\Psi_{\mu,\nu}$ is equal to the number of factors that appear in the double products in~\eqref{eq:defPsi}, which in turn is equal to the number of odd hooks in the partition $\Phi(\mu,\nu,0)$ by Lemma~\ref{lem:formulahooklengths0}. As described in the proof of Claim 2, there are $\lvert \mu \rvert + \lvert \nu \rvert$ such terms. This establish the degree statement. Next, one immediately observes that the leading coefficient of~$\Psi_{\mu,\nu}$ is $(-1)^\gamma \, 2^{\lvert \mu \rvert + \lvert \nu \rvert}$, with
\begin{equation}\label{eq:pwithnu}
	\gamma=h_{\Phi(\mu,\nu,0)} + \# \left\{ (m,n) \in \mathbb{Z}^2 \mid m\in M_\mu+s, n\not\in M_\nu + s', n<m\right\}.
\end{equation} 
Hence it is sufficient to show that $\gamma$ and $\lvert \nu \rvert$ are equal in parity. We prove this in a similar fashion as we proved Claim 1, that is by induction on $\lvert (\mu,\nu) \rvert$.

If $(\mu,\nu)=(\emptyset,\emptyset)$, then $\lvert \nu \rvert=\gamma=0$ and hence the result certainly holds. Now suppose that $(\tilde{\mu},\nu)\lessdot (\mu,\nu)$ and that the claim holds for $(\tilde{\mu},\nu)$. Then the transition from $(\tilde{\mu},\nu)$ to $(\mu,\nu)$ is represented by a dot moving from a position $t$ to a position $t+1$ in the Maya diagram of~$\mu$. It is easy to see that both terms in~\eqref{eq:pwithnu} change parity in this transition if and only if $t\not\in M_\nu+s'$. Hence $\gamma$ does not change in parity. Since $\lvert \nu \rvert$ does not change at all, the claim also holds for~$(\mu,\nu)$.

Now suppose that $(\mu,\tilde{\nu})\lessdot (\mu,\nu)$ and that the claim holds for $(\mu,\tilde{\nu})$. Clearly, $\lvert \tilde{\nu}\rvert$ and $\lvert \nu \rvert$ differ in parity. If the transition from $(\mu,\tilde{\nu})$ to $(\mu,\nu)$ is represented by a dot moving from position $t$ to $t+1$ in the second Maya diagram, then it is easy to see that the first term in~\eqref{eq:pwithnu} changes parity if and only if $t+1 \not\in M_\nu+s$, and the second term if and only if $t+1\in M_\nu+s$. Therefore~$\gamma$ changes parity, so again $\lvert \nu \rvert$ and $\gamma$ are equal in parity. 
\end{proof}

\subsection{The subleading coefficient in terms of the content}\label{sec:coeffsubleading}
By definition $R_{\lambda}$ is a monic polynomial of degree $|\mu|+|\nu|$. If we set $j=1$ in~\eqref{eq:coefficients}, we obtain a formula for the subleading coefficient $r_{\lambda,1}$. However, there is another appealing expression involving partition data. For this, we note that in the standard literature on integer partitions~\cite{MacDonald}, the \textbf{content} of each box $(i,j)$ in the Young diagram of $\lambda$ is defined as $j - i$. We denote the sum of contents over all boxes by
\begin{equation}\label{eq:contents}
c(\lambda) = \sum_{(i,j)\in\lambda}(j-i)=\frac{1}{2} \sum_{i=1}^{\ell(\lambda)} \lambda_i(\lambda_i-(2i-1))
\end{equation}  
where the last equality follows directly from summing over rows. We have the following result and visualize some examples in Figure~\ref{fig:r1}.

\begingroup
\setlength{\tabcolsep}{40pt}
\renewcommand{\arraystretch}{1.5}
\begin{figure}[t]
	\centering
	\begin{tabular}{c c c}
		
		$\lambda=(4,2,1)$
		&$\lambda=(5,4,2)$ 
		&$\lambda=(2,2,2)$
		\\
		
		\begin{ytableau}
			$0$ & $1$ & $2$ & $3$ \\
			$-1$ & $0$ \\
			
		\end{ytableau}
		&\begin{ytableau}
			$0$ & $1$ & $2$ & $3$ & $4$ \\
			$-1$ & $0$  & $1$ & $2$ \\
			$-2$ & $-1$ 
		\end{ytableau} 
		&\begin{ytableau}
			$0$ & $1$ \\
			$-1$ & $0$ \\
			$-2$  & $-1$ 
		\end{ytableau}
		\\
		
		$r_{(4,2,1),1} = -5$
		&$r_{(5,4,2),1} = -9$
		&$r_{(2,2,2),1} = 3$
	\end{tabular}
	\caption{Examples of the values $r_{\lambda,1}$ where the boxes of the Young diagrams contain the respective contents.}
	\label{fig:r1}
\end{figure}
\endgroup

\begin{proposition}\label{prop:subleadingcoeff}
For any partition $\lambda$ with non-empty quotient, we have that the subleading coefficient of the polynomial $R_\lambda$, defined in~\eqref{eq:WHdecomposition} and~\eqref{eq:Rcoeff}, is given by 
\begin{equation}\label{eq:subleadingcoeff}
	r_{\lambda,1}
		= -c(\lambda)
\end{equation}
where $c(\lambda)$ is the sum of contents defined in~\eqref{eq:contents}.
\end{proposition}
\begin{proof}
Setting $j=1$ in~\eqref{eq:recurrencecoefficients}, 
dividing both sides by $F_{\lambda}$ and using 
 $r_{\tilde{\lambda},0}=1$ yields
\begin{equation}\label{eq:r1}
	r_{\lambda,1}
		= -\frac{|\lambda|(|\lambda|-1)}{2} \sum_{\tilde{\lambda}} (-1)^{\htt_2(\lambda/\tilde{\lambda})} \frac{F_{\tilde{\lambda}}}{F_{\lambda}}
\end{equation}
where  the sum runs over all partitions $\tilde{\lambda}$ that are obtained by removing a domino tile from $\lambda=(\lambda_1,\lambda_2,\dots,\lambda_{\ell(\lambda)})$. The main idea of the proof is to expand the fraction $F_{\tilde{\lambda}}/F_{\lambda}$ using the hook formulae in~\eqref{eq:Flambdahooklengths} written using 
the degree vector ${n_{\lambda}=(n_1,n_2,\dots,n_{\ell(\lambda)})}.$ 

If $\lambda_m\geq\lambda_{m+1}+2$, then a horizontal domino can be removed from row $m$ of $\lambda$. In this case 
\begin{equation}\label{eq:r1proof1}
	\frac{F_{\tilde{\lambda}}}{F_{\lambda}} 
		= \frac{1}{|\lambda|(|\lambda|-1)} n_m(n_m-1) \prod_{\substack{i=1 \\ i \neq m}}^{\ell(\lambda)} \frac{n_m-n_i-2}{n_m-n_i}
\end{equation}
and we note that $(-1)^{\htt(\lambda/\tilde{\lambda})}=1$. If $\lambda_m=\lambda_{m+1}>\lambda_{m+2}$, then a vertical domino can be removed from row $m$ and $m+1$ of $\lambda$. We then find
\begin{equation}\label{eq:r1proof2}
	\frac{F_{\tilde{\lambda}}}{F_{\lambda}} 
		= \frac{-1}{|\lambda|(|\lambda|-1)} n_m(n_m-1) \prod_{\substack{i=1 \\ i \neq m}}^{\ell(\lambda)} \frac{n_m-n_i-2}{n_m-n_i}
\end{equation}
and $(-1)^{\htt(\lambda/\tilde{\lambda})}=-1$. Hence if we plug~\eqref{eq:r1proof1} and~\eqref{eq:r1proof2} into~\eqref{eq:r1}, we get
\begin{equation}\label{eq:r1proof3}
	r_{\lambda,1}
		= -\frac{1}{2} \sum_{m=1}^{\ell(\lambda)} n_m(n_m-1) \prod_{\substack{i=1 \\ i \neq m}}^{\ell(\lambda)} \frac{n_m-n_i-2}{n_m-n_i}
\end{equation}
where it is possible to sum from $m=1$ until $\ell(\lambda)$ because we can either remove a horizontal domino in row $m$, or we can remove a vertical domino in rows $m$ and $m+1$, or the term in the sum vanishes because it is not possible to remove any dominoes. 

Finally, we argue that~\eqref{eq:r1proof3} equals $-c(\lambda)$. To show this, we use the last expression for $c(\lambda)$ in~\eqref{eq:contents} and write it in terms of the degree vector via $\lambda_i=n_i-\ell(\lambda)+i$ for $i=1,2,\dots,\ell(\lambda)$. We find
\begin{equation}\label{eq:contentdegreevector}
	-c(\lambda)
		= -\frac{1}{2} \sum_{i=1}^{\ell(\lambda)} (n_i-\ell(\lambda)+i)(n_i-\ell(\lambda)-i+1).
\end{equation}
We now use the combinatorial identity in Lemma~\ref{lem:identity} from the appendix, setting $(x_1,x_2,\dots,x_n)$ equal to the degree vector of $\lambda$. We therefore find that the expression for $r_{\lambda,1}$ in~\eqref{eq:r1proof3} equals the expression~$-c(\lambda)$ in \eqref{eq:contentdegreevector}.
\end{proof}

\begin{remark}
An alternative way to derive the result $r_{\lambda,1 }=-c(\lambda)$ is via the expression \eqref{eq:WHcoeff}. We set $j=1$ to obtain 
\begin{equation*}
	F_{\lambda} r_{\lambda,1}
		= - \binom{|\lambda|}{2} a(\lambda,1)
\end{equation*}
where $a(\lambda,1)$ represent the character value of the conjugacy class of the cycle type $(2,1^{|\lambda|-2})$. This character value can be expressed explicitly in terms of the partition entries (see for example~\cite[I.7,~Ex.~7]{MacDonald}). Combining that result with \cite[I.1,~Ex.~3]{MacDonald} yields the expression
\begin{equation*}
	a(\lambda,1) = \binom{|\lambda|}{2}^{-1} \,F_{\lambda} \,c(\lambda)
\end{equation*}
and so we again obtain $r_{\lambda,1 }=-c(\lambda)$.
\end{remark}

\begin{remark}
Box $(i,j)$ belongs to the Young diagram of $\lambda$ if and only if box $(j,i)$ belongs to the Young diagram of the conjugate partition $\lambda'$. Therefore,~\eqref{eq:subleadingcoeff} trivially shows  that if $\lambda=\lambda'$, then $r_{\lambda,1}=0$, in agreement with Corollary~\ref{cor:Rlambda}. The converse is not true: there are partitions $\lambda$ that are not self-conjugate but 
the corresponding Wronskian Hermite polynomial has  $r_{\lambda,1}=0$. The smallest examples of such partitions are of size 15, one of which is $\lambda=(5^2,2,1^3)$. 
\end{remark}

We now rewrite Proposition~\ref{prop:subleadingcoeff} in terms of the core and quotient of the partition. 

\begin{proposition}\label{prop:contentquotient}
For any partition $\lambda$ with non-empty quotient $(\mu,\nu)$ and core ${(k,k-1,\dots,2,1)}$, we have that the subleading coefficient of the polynomial $R_\lambda$, defined in~\eqref{eq:WHdecomposition} and~\eqref{eq:Rcoeff}, is given by
\begin{equation}\label{eq:contentquotient}
	r_{\lambda,1}
		= (|\mu|-|\nu|)(2k+1) + 4(c(\mu)+c(\nu))
\end{equation}
where $c(\mu)$ and $c(\nu)$ are the sum of the contents of the quotient partitions respectively. 
\end{proposition}
\begin{proof}
Proposition~\ref{prop:subleadingcoeff} states that $r_{\lambda,1}=-c(\lambda)$. We express $-c(\lambda)$ in terms of the degree vector, see \eqref{eq:contentdegreevector}, and then expand to obtain 
\begin{equation}\label{eq:contentdegreevectorexpanded}
	r_{\lambda,1}
		=-\frac{1}{2} \sum_{i=1}^{\ell(\lambda)} n_i^2
		+ \left(\ell(\lambda)-\frac{1}{2}\right) \sum_{i=1}^{\ell(\lambda)} n_i		
		- \frac{\ell(\lambda)(\ell(\lambda)-1)(2\ell(\lambda)-1)}{6}.
\end{equation}
By the definition of the core and quotient as given in Section~\ref{sec:coresandquotients}, one can easily write the elements of the degree vector~$n_\lambda$ explicitly in terms of $k$ and the elements in the degree vectors~$n_\mu$ and~$n_\nu$. Doing this for all $n_i$ in \eqref{eq:contentdegreevectorexpanded} yields an expression for $r_{\lambda,1}$ in terms of the quotient $(\mu,\nu)$ and the integer $k$. If one also uses \eqref{eq:contentdegreevector} to expand $c(\mu)$ and $c(\nu)$ in the right-hand side of \eqref{eq:contentquotient}, then the result follows from elementary rewriting. 
\end{proof}

\begin{remark}
Note that identity~\eqref{eq:contentquotient} is invariant under replacing $(\mu,\nu,k)$ with $(\nu,\mu,-k-1)$. This is a direct consequence of the construction of the core and quotient in Section~\ref{sec:coresandquotients}. Also observe that the coefficient $r_{\lambda,1}$ is linear in $k$ and is  constant if $|\mu|=|\nu|$. This agrees with Theorem \ref{thm:coeffpoly}. Moreover, numerical calculations suggest that if $|\mu|=|\nu|$, then $r_{\lambda,|\mu|+|\nu|-1}=0$ for all $k$.
\end{remark}

\begin{remark}
The result of Proposition~\ref{prop:subleadingcoeff} can be extended to the Wronskian Appell polynomials introduced in~\cite{Bonneux_Hamaker_Stembridge_Stevens}. The details are given in Appendix~\ref{sec:leadingcoefficientsWAP}.
\end{remark}
  

\section{Asymptotic behaviour for a fixed quotient}\label{sec:asymptotics}
If one fixes the quotient $(\mu,\nu)$ and  writes $\Phi(\mu,\nu,k)$ for the partition with quotient $(\mu,\nu)$ and core $(k,k-1,\dots,2,1)$, then Theorem~\ref{thm:coeffpoly} states that the coefficients of the remainder polynomial~$R_{\Phi(\mu,\nu,k)}$ are polynomials in $k$. 
In Section~\ref{sec:asymptoticpolynomial} this allows us  to deduce the asymptotic behaviour of the remainder polynomial. In turn, this leads to the investigation of the asymptotic behaviour of the zeros given in Section~\ref{sec:asymptoticzeros}.

\subsection{Asymptotic behaviour of the remainder polynomial}\label{sec:asymptoticpolynomial}
If one fixes the quotient of a partition and lets the size of the core grow, one has the following asymptotic behaviour of the associated remainder polynomial.

\begin{theorem}\label{thm:asymptotics}
Fix a pair of partitions $(\mu,\nu)$ and let $\Phi(\mu,\nu,k)$ denote the partition with quotient $(\mu,\nu)$ and core $(k,k-1,\dots,2,1)$. Then
\begin{equation}\label{eq:asymptotics}
	\lim_{k \to + \infty} \frac{R_{\Phi(\mu,\nu,k)}(2kx)}{(2k)^{|\mu|+|\nu|}}
		= (x+1)^{|\mu|} \, (x-1)^{|\nu|}
\end{equation}
uniformly for $x$ in compact subsets of the complex plane, where the speed of convergence is~$O(k^{-1})$.
\end{theorem}

The main ingredient required for the proof of this asymptotic result is the following lemma. For this, recall from Theorem \ref{thm:coeffpoly} that the coefficient $r_{\Phi(\mu,\nu,k),j}$ is a polynomial in $k$ of degree at most $j$. 

\begin{lemma}\label{lem:leadingcoeffr}
Take a pair of partitions $(\mu,\nu)$ and let $0\leq j \leq \lvert \mu \rvert +\lvert \nu \rvert$. Then the coefficient of~$k^j$ in $r_{\Phi(\mu,\nu,k),j}$ is equal to
\begin{equation}\label{eq:leadingcoeffr}
	2^j \sum_{l=0}^j (-1)^l \binom{\lvert \mu \rvert}{j-l} \binom{\lvert \nu \rvert}{l}.
\end{equation}
\end{lemma}
\begin{proof}
If we combine~\eqref{eq:coefficientsR0} with~\eqref{eq:psiR}, we obtain that
\begin{equation}\label{eq:rpsi}
	r_{\Phi(\mu,\nu,k),j}
		=
		\binom{|\mu|+|\nu|}{j} 
		\sum_{(\tilde{\mu},\tilde{\nu})<_{j}(\mu,\nu)}
		\frac{F^{(2)}_{\tilde{\mu},\tilde{\nu}} F^{(2)}_{\mu/\tilde{\mu},\nu/\tilde{\nu}}}{F^{(2)}_{\mu,\nu}} 
		\,
		\frac{\Psi_{\mu,\nu}(k)}{\Psi_{\tilde{\mu},\tilde{\nu}}(k)}.
\end{equation}
Since $\Psi_{\mu,\nu}$ is a polynomial of degree $\lvert \mu \rvert + \lvert \nu \rvert$ with leading coefficient $(-1)^{\lvert \nu \rvert} 2^{\lvert \mu \rvert +\lvert \nu \rvert}$, see Proposition~\ref{prop:psi}, we conclude that if $(\tilde{\mu},\tilde{\nu})<_{j}(\mu,\nu)$, then
\begin{equation}\label{eq:fracpsi}
\frac{\Psi_{\mu,\nu}(k)}{\Psi_{\tilde{\mu},\tilde{\nu}}(k)}=(-1)^{|\nu|-|\tilde{\nu}|} (2k)^j + O(k^{j-1}).
\end{equation}
Note that when $(\tilde{\mu},\tilde{\nu})<_{j}(\mu,\nu)$, there is an $0\leq l\leq j$ such that $\tilde{\mu} <_{j-l} \mu$ and $\tilde{\nu} <_{l} \nu$. Combining this with~\eqref{eq:fracpsi}, we expand~\eqref{eq:rpsi} as
\begin{align*}
r_{\Phi(\mu,\nu,k),j}
	&= \binom{|\mu|+|\nu|}{j} \sum_{l=0}^j \sum_{\tilde{\mu}<_{j-l} \mu} \sum_{\tilde{\nu}<_l \nu} \frac{\binom{\lvert \tilde{\mu} \rvert + \lvert \tilde{\nu} \rvert}{\lvert \tilde{\mu} \rvert} \binom{j}{l}}{\binom{\lvert \mu\rvert +\lvert \nu \rvert}{\lvert \mu \rvert}} \frac{F_{\tilde{\mu}} F_{\tilde{\nu}} F_{\mu/\tilde{\mu}} F_{\nu/\tilde{\nu}}}{F_\mu F_\nu} \left((-1)^l (2k)^j +O(k^{j-1})\right), \\
	&=\sum_{l=0}^j \left(\sum_{\tilde{\mu}<_{j-l} \mu} \frac{F_{\tilde{\mu}} F_{\mu/\tilde{\mu}}}{F_\mu}\right) \left(\sum_{\tilde{\nu}<_{l} \nu} \frac{F_{\tilde{\nu}} F_{\nu/\tilde{\nu}}}{F_\nu}\right) \binom{\lvert \mu\rvert}{j-l} \binom{\lvert \nu\rvert}{l} \left((-1)^l (2k)^j +O(k^{j-1})\right), \\
	&=\sum_{l=0}^j \binom{\lvert \mu\rvert}{j-l} \binom{\lvert \nu\rvert}{l} \left((-1)^l (2k)^j +O(k^{j-1})\right)
\end{align*}
where in the first equality we have used~\eqref{eq:F2withF1} three times, in the second equality we have rearranged factors and combined the four binomial coefficients into two, and in the last equality we observed that the sums over $\tilde{\mu}$ and $\tilde{\nu}$ are both trivially equal to one. Now using the fact that $r_{\Phi(\mu,\nu,k),j}$ is a polynomial in $k$ by Theorem~\ref{thm:coeffpoly}, we find that the coefficient of $k^j$ in $r_{\Phi(\mu,\nu,k),j}$ is indeed given by~\eqref{eq:leadingcoeffr}.
\end{proof}

We now come to the proof of the asymptotic result.

\begin{figure}[t]
	\centering
	\includegraphics[height=6cm]{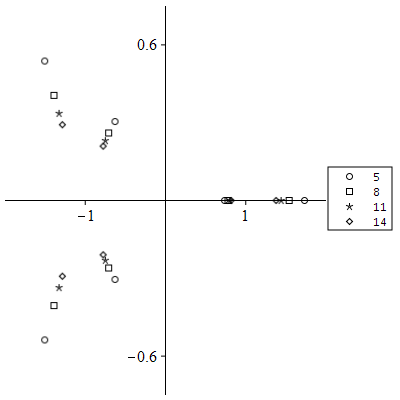}
	\qquad
	\includegraphics[height=6cm]{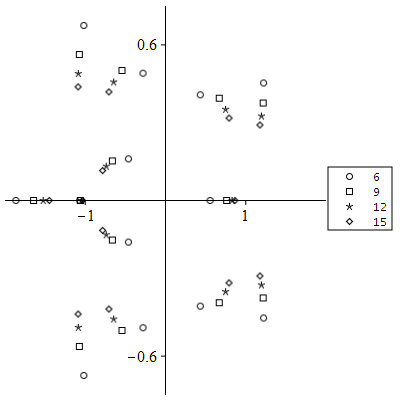}
	\caption{Zeros of the polynomials $R_{\Phi(\mu,\nu,k)}(2kx)$ with $\mu=(2^2)$, $\nu=(2)$ on the left and $\mu=(4,3,1)$, $\nu=(2^2,1)$ on the right for various core sizes. The core sizes $k$ are indicated in the legends.}
	\label{fig:AsymptoticBehaviorExample}
\end{figure}

\begin{proof}[Proof of Theorem~\ref{thm:asymptotics}]
Expanding the remainder polynomial using~\eqref{eq:Rcoeff} gives
\begin{equation}\label{eq:proofasymptotics}
	\frac{R_{\Phi(\mu,\nu,k)}(2kx)}{(2k)^{|\mu|+|\nu|}}
		= \sum_{j=0}^{|\mu|+|\nu|} \frac{r_{\Phi(\mu,\nu,k),j}}{(2k)^j} x^{|\mu|+|\nu|-j}.
\end{equation}
By Theorem~\ref{thm:coeffpoly} we know that $r_{\Phi(\mu,\nu,k),j}$ is a polynomial in $k$ of degree at most $j$, and Lemma~\ref{lem:leadingcoeffr} established that the coefficient of $k^j$ is given by~\eqref{eq:leadingcoeffr}. Plugging this into~\eqref{eq:proofasymptotics} yields 
\begin{equation*}
	\frac{R_{\Phi(\mu,\nu,k)}(2kx)}{(2k)^{|\mu|+|\nu|}}
		= \sum_{j=0}^{|\mu|+|\nu|} \left(\sum_{l=0}^j (-1)^l \binom{\lvert \mu \rvert}{j-l} \binom{\lvert \nu \rvert}{l}\right) x^{|\mu|+|\nu|-j} + O(k^{-1}).
\end{equation*}
Hence
\begin{equation*}
	\lim_{k\rightarrow \infty} \frac{R_{\Phi(\mu,\nu,k)}(2kx)}{(2k)^{|\mu|+|\nu|}}
		= \sum_{j=0}^{|\mu|+|\nu|} \left(\sum_{l=0}^j (-1)^l \binom{\lvert \mu \rvert}{j-l} \binom{\lvert \nu \rvert}{l}\right) x^{|\mu|+|\nu|-j}
\end{equation*}
uniformly for $x$ in compact subsets of the complex plane. The required results follows from noting that the right-hand side is the expansion of the polynomial $(x+1)^{\lvert \mu \rvert}(x-1)^{\lvert \nu \rvert}$.
\end{proof}

\begin{remark}
Recall that the map $\Phi$ is defined for all tuples $(\mu,\nu,k)\in \mathbb{Y}\times \mathbb{Y} \times \mathbb{Z}$ in \eqref{eq:Phi} including negative values of $k.$  This gives a way to consider `negative core lengths', which is best understood in terms of the relation $\Phi(\mu,\nu,k)=\Phi(\nu,\mu,-k-1)$. This also provides an asymptotic result for $k\rightarrow -\infty$, via Theorem~\ref{thm:asymptotics}. 
\end{remark} 

\begin{figure}[t]
	\centering
	\includegraphics[height=6cm]{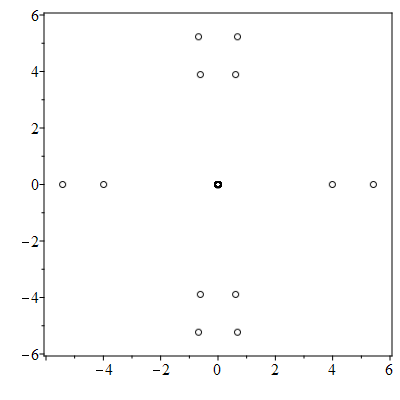}
	\qquad
	\includegraphics[height=6cm]{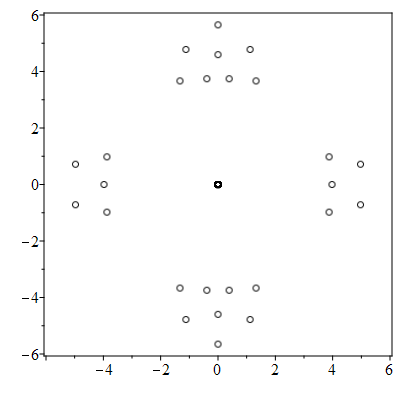}
	\caption{Zeros of the polynomials $\He_{\Phi(\mu,\nu,10)}(x)$ with $\mu=(2^2)$, $\nu=(2)$ on the left and $\mu=(4,3,1)$, $\nu=(2^2,1)$ on the right.}
	\label{fig:AsymptoticBehaviorExamplebis}
\end{figure}

\subsection{Asymptotic behaviour of the zeros}
\label{sec:asymptoticzeros}
The asymptotic behaviour~\eqref{eq:asymptotics} implies that there are $|\mu|$ zeros attracted to $-1$ and $|\nu|$ zeros attracted to 1 as the size of the core tends to infinity. In Figure~\ref{fig:AsymptoticBehaviorExample} we see the zeros of the rescaled remainder polynomial approaching $-1$ and $1$ as $k$ grows for two fixed quotients. Also for fixed finite (but large enough) $k$, we observe that the zeros exhibit a suprisingly-simple structure. Namely, if we take a quotient $(\mu,\nu)$ and fix $k\geq \ell(\mu')+\ell(\nu)-1$, then the $\lvert \mu \rvert + \lvert \nu \rvert$ zeros of $R_{\Phi(\mu,\nu,k)}$ can be grouped in `vertical strings', which is best explained by considering an example. Take $\mu=(4,3,1)$ and $\nu=(2^2,1)$ such that $\mu'=(3,2^2,1)$ and $\nu'=(3,2)$. Then, if $k\geq \ell(\mu')+\ell(\nu)-1=6$, we observe that the zeros in the left half-plane come in $\ell(\mu)=3$ vertical strings; groups that roughly have the same real part. The numbers of zeros within these strings are precisely 4, 3, and 1, i.e., the parts of $\mu$, see Figure \ref{fig:AsymptoticBehaviorExample} on the right where the zeros are plotted for several large enough values of $k$. Similarly, in the right half-plane, we see $\ell(\nu')=2$ such vertical strings. One of these strings contains 3 zeros, the other 2; these are precisely the parts of $\nu'$.

We conjecture that for a general quotient $(\mu,\nu)$ and $k\geq \ell(\mu')+\ell(\nu)-1$, the zeros of $R_{\Phi(\mu,\nu,k)}$ exhibit the same behaviour, i.e., there are $\ell(\mu)$ vertical strings of zeros that are attracted to~$-1$, with the $i^\mathrm{th}$ string containing $\mu_i$ zeros for $1 \leq i \leq \ell(\mu)$, and $\ell(\nu')$ vertical strings of zeros that are attracted to 1, with the $i^\mathrm{th}$ string  containing $\nu'_i$ zeros for $1\leq i \leq \ell(\nu')$. We observe that the ordering of the strings may not match the ordering of the parts in the partitions.

Our trade-off value $k=\ell(\mu')+\ell(\nu)-1$ is found by examining several examples and we offer the following heuristic justification. If one fixes a quotient $(\mu,\nu)$, then the  Maya diagrams $M_\mu$ and $M_\nu$ are fixed. However, if $k$ increases, then this corresponds to shifting the Maya diagram of $\nu$ to the right. For the specific value $k=\ell(\mu')+\ell(\nu)-1$, the first empty box of $M_\nu$ is precisely below the last filled box of $M_\mu$. From this point onwards, the contributions of the quotient $(\mu,\nu)$ to the partition seem to become `independent'. We do not have a rigorous proof for this justification, but the checked examples do confirm this trade-off value.

The above conjecture translates to a conjecture about the zeros of Wronskian Hermite polynomials by the relation in Theorem \ref{thm:WHdecomposition}. Figure~\ref{fig:AsymptoticBehaviorExamplebis} shows the unscaled zeros of the Wronskian Hermite polynomials corresponding to the examples in Figure \ref{fig:AsymptoticBehaviorExample} where $k=10$. By Theorem~\ref{thm:WHdecomposition}, the number of zeros doubles with respect to Figure \ref{fig:AsymptoticBehaviorExample}, but the behaviour is similar. The zeros near the positive imaginary axes form $\ell(\mu)$ `horizontal' strings with each string having $\mu_i$ approximately-equal imaginary parts, while the zeros near the positive real axis form vertical strings with $\nu'_i$ parts. We intend to explore this curious `physical' appearance of the Young diagrams of the quotient partitions in such zero plots in future work.

We also note that a different qualitative relationship between the Young diagram of specific even partitions and the zero distributions of $\He_{\lambda}$ was noted in~\cite[Section~3]{Felder_Hemery_Veselov}. The authors considered Wronskian Hermite polynomials associated to even partitions of the form $\lambda=((2\lambda_1)^2, (2\lambda_2)^2, \dots, (2\lambda_r)^2)$. However, such partitions have empty core $(k=0)$, so our asymptotic result and large-$k$ observations do not add further insight.

The question of the location of the zeros of the Wronskian Hermite polynomials is both interesting in its own right, as well as for applications, and dates back to at least the 1960s~\cite{Karlin_Szego}. More recently, the asymptotic behaviour of the zeros of the exceptional Hermite polynomials has been studied in~\cite{Kuijlaars_Milson}. A further motivation for studying the locations of the zeros  is that the rational solutions of the fourth Painlev\'e equation are written as ratios of certain Wronskian Hermite polynomials \cite{Airault,Bassom_Clarkson_Hicks,Murata,Noumi_Yamada,Okamoto}. The zeros in the generalized Okamoto and generalized Hermite cases, which feature in the Painlev\'e IV rational solutions, were shown by Clarkson~\cite{Clarkson-zeros,Clarkson-PIV} to form highly-regular patterns in the complex plane. Some rigorous results on the distribution of the zeros of the generalized Hermite polynomials and the generalized Okamoto polynomials in various asymptotic regimes have appeared recently~\cite{Buckingham,Masoero_Roffelsen,Masoero_Roffelsen-2019,Novokshenov_Shchelkonogov}, complementing the asymptotic results for small-length Wronskians obtained in~\cite{Felder_Hemery_Veselov,GarciaFerrero_GomezUllate}.  

Felder et al.~\cite{Felder_Hemery_Veselov} conjectured the Wronskian Hermite polynomials have no real zeros if and only if the associated partition is even, and made a similar statement about the number of purely imaginary zeros. The result about the real zeros was (already) proven in a more general context independently by Krein~\cite{Krein} and Adler~\cite{Adler}, and recently generalized in~\cite{GarciaFerrero_GomezUllate}. In addition, from the latter paper it follows that if we assume that the Hermite setting is semi-degenerate, specifically that common zeros between certain Wronskian polynomials only occur at the origin, then the number of real zeros can be explicitly stated in terms of the associated partition.


\section{Connection with Laguerre polynomials}\label{sec:connectionwithlaguerre}
In this section we explain how Wronskian Hermite polynomials can be seen as discrete versions of Wronskians involving Laguerre polynomials, generalizing the well-known relation between Hermite and Laguerre polynomials. 
For all $n\geq0$, we note from~\eqref{eq:recurrenceHermite} that 
\begin{equation*}
	\He_n(x) 
		:= 2 ^{-\frac{n}{2}} H_n\left(\frac{x}{\sqrt{2}}\right)
\end{equation*}
where $H_n$ denotes the classical definition for Hermite polynomials, as, for example, given in~\cite{Szego}. Likewise, we use modified Laguerre polynomials $\hat{L}_n^{(\alpha)}$, which we define as 
\begin{equation}\label{eq:monicLaguerre}
	\hat{L}_n^{(\alpha)}(x)
		:= (-1)^n \, n! \, L_n^{(\alpha)}(x)
\end{equation}
for all $n\geq 0$ and where $L_n^{(\alpha)}$ denotes the classical definition of the Laguerre polynomial~\cite{Szego}. In this way, both $\He_n$ and $\hat{L}_n^{(\alpha)}$ are monic polynomials of degree $n$, and they are related according to
\begin{equation}\label{eq:HermiteLaguerre}
\He_{2n}(x)
= 2^n \, \hat{L}_n^{(-\frac{1}{2})}\left(\frac{x^2}{2}\right)
\qquad
\qquad
\He_{2n+1}(x)
= 2^n \, x \, \hat{L}_n^{(\frac{1}{2})}\left(\frac{x^2}{2}\right)	
\end{equation}
for all $n\geq0$; see~\cite[Formula~(5.6.1)]{Szego} for the identities in terms of the classical definitions.

We now introduce Wronskians involving Laguerre polynomials following~\cite{Bonneux_Kuijlaars,Duran-Laguerre,Duran_Perez,GomezUllate_Grandati_Milson-L+J}. For any two partitions $\mu$ and $\nu$ with degree vectors $n_{\mu}$ and $m_{\nu}$, and for any parameter $\alpha\in\mathbb{R}$ such that the values 
\begin{equation}\label{eq:parameter}
	n_1,n_2,\dots,n_{\ell(\mu)},m_1-\alpha,m_2-\alpha,\dots,m_{\ell(\nu)}-\alpha
\end{equation}
are pairwise different, we define the polynomial
\begin{equation}\label{eq:WL}
	\hat{L}_{\mu,\nu}^{(\alpha)}
		:=  \frac{1}{\Delta(n_{\mu},m_\nu-\alpha)} 
		x^{(\ell(\mu)+\alpha)\ell(\nu)} \Wr[f_1,f_2,\dots,f_{\ell(\mu)},g_1,g_2,\dots,g_{\ell(\nu)}]
\end{equation}
where
\begin{align*}
f_j(x)
&= \hat{L}_{n_j}^{(\alpha)}(x),
&& j=1,2,\dots,\ell(\mu), 
\\
g_{j}(x)
&= x^{-\alpha}\hat{L}_{m_j}^{(-\alpha)}(x),
&& j=1,2,\dots,\ell(\nu).
\end{align*}
Here $\Delta(n_{\mu},m_\nu-\alpha)$ denotes the Vandermonde determinant of the elements given in~\eqref{eq:parameter}. We have that $\hat{L}_{\mu,\nu}^{(\alpha)}$ is a monic polynomial of degree $|\mu|+|\nu|$ by~\cite[Proposition~3.1]{Bonneux}. The polynomial  $\hat{L}_{\mu,\nu}^{(\alpha)}$ may be defined at each of the points disallowed by \eqref{eq:parameter} by taking a relevant limit as a function of~$\alpha$.

The relations in~\eqref{eq:HermiteLaguerre}  extend to Wronskian polynomials, and it is now natural to express such identities in terms of cores and quotients. In particular, if $\lambda=(n)$ then using its core and quotient as stated in Lemma~\ref{lem:CoreAndQuotientTrivialPartition}, we find~\eqref{eq:WHermiteWLaguerre}  reduces to~\eqref{eq:HermiteLaguerre}.

\begin{proposition}\label{prop:WHermiteWLaguerre}
Let $\lambda$ be a partition with core $(k,k-1,\dots,2,1)$ and quotient $(\mu,\nu)$. Then 
\begin{equation}\label{eq:RLaguerre}
	R_{\lambda}(x)	
		= 2^{|\mu|+|\nu|} \, \hat{L}_{\mu,\nu}^{(\alpha_k)}\left(\frac{x}{2}\right)
\end{equation}
with $\alpha_k:=-1/2-\ell(\mu)+\ell(\nu)-k.$ In other words, we have
\begin{equation}\label{eq:WHermiteWLaguerre}
	\He_{\lambda}(x)
		= 2^{|\mu|+|\nu|}  x^{\frac{k(k+1)}{2}} \hat{L}_{\mu,\nu}^{(\alpha_k)}\left(\frac{x^2}{2}\right).
\end{equation}
\end{proposition}
\begin{proof}
It is sufficient to prove~\eqref{eq:WHermiteWLaguerre}, since~\eqref{eq:RLaguerre} then immediately follows from~\eqref{eq:WHdecomposition}. 

From the definition of $(\mu,\nu)$ as the quotient of $\lambda$, there exist $t_1\geq 0$ and $t_2\geq 0$ such that the Maya diagram $M_\lambda$ is equivalent to the 2-modular decomposition of the Maya diagrams $\widetilde{M}_\mu$ and $\widetilde{M}_\nu$, where $\widetilde{M}_\mu=M_\mu + t_1$ and $\widetilde{M}_\nu=M_\nu + t_2$ are equivalent to $M_\mu$ and $M_\nu$, respectively. From~\eqref{eq:shift} we have 
\begin{equation*}
	\ell(\nu)+t_2-\ell(\mu)-t_1 = k.
\end{equation*}
This implies that $\alpha_k=-1/2+t_1-t_2$. By Theorem 1 in~\cite{Bonneux_Kuijlaars}, we then have that (up to the normalization constant) $\hat{L}_{\mu,\nu}^{(\alpha_k)}(x)$ is equal to
\begin{equation}\label{eq:WLshiftedMaya}
	x^{(\ell(\mu)+t_1-\frac{1}{2})(\ell(\nu)+t_2)} \Wr[\tilde{f}_1,\tilde{f}_2,\dots,\tilde{f}_{\ell(\mu)+t_1},\tilde{g}_1,\tilde{g}_2,\dots,\tilde{g}_{\ell(\nu)+t_2}]
\end{equation}
where
\begin{align*}
\tilde{f}_j(x)
&= \hat{L}_{n_j+t_1}^{(-\frac{1}{2})}(x),
&& j=1,2,\dots,\ell(\mu), 
\\
\tilde{f}_j(x)
&= \hat{L}_{\ell(\mu)+t_1-j}^{(-\frac{1}{2})}(x),
&& j=\ell(\mu)+1,\ell(\mu)+2,\dots,\ell(\mu)+t_1,
\\
\tilde{g}_{j}(x)
&= x^{\frac{1}{2}}\hat{L}_{m_j+t_2}^{(\frac{1}{2})}(x),
&& j=1,2,\dots,\ell(\nu), 
\\
\tilde{g}_{j}(x)
&= x^{\frac{1}{2}}\hat{L}_{\ell(\nu)+t_2-j}^{(\frac{1}{2})}(x),
&& j=\ell(\nu)+1,\ell(\nu)+2,\dots,\ell(\nu)+t_2.
\end{align*}
Note that the degrees that appear for the functions $\tilde{f}_j$ and $\tilde{g}_j$ are precisely the non-negative locations of the dots in the Maya diagrams $\widetilde{M}_\mu$ and $\widetilde{M}_\nu$, respectively. In this way, Theorem 1 in~\cite{Bonneux_Kuijlaars} tells us precisely how to account for shifting the origin of Maya diagrams when working with Wronskians involving Laguerre polynomials. If we evaluate the functions $\tilde{f}_j$ and $\tilde{g}_j$ at $x^2/2$ and use~\eqref{eq:HermiteLaguerre}, we obtain 
\begin{align*}
\tilde{f}_j\left(\frac{x^2}{2}\right)
&= 2^{-n_j-t_1} \He_{2(n_j+t_1)}(x),
&& j=1,2,\dots,\ell(\mu), 
\\
\tilde{f}_j\left(\frac{x^2}{2}\right)
&= 2^{-\ell(\mu)-t_1+j} \He_{2(\ell(\mu)+t_1-j)}(x),
&& j=\ell(\mu)+1,\ell(\mu)+2,\dots,\ell(\mu)+t_1,
\\
\tilde{g}_j\left(\frac{x^2}{2}\right)
&= 2^{-m_j-t_2-\frac{1}{2}} \He_{2(m_j+t_2)+1}(x),
&& j=1,2,\dots,\ell(\nu), 
\\
\tilde{g}_j\left(\frac{x^2}{2}\right)
&= 2^{-\ell(\nu)-t_2+j-\frac{1}{2}} \He_{2(\ell(\nu)+t_2-j)+1}(x),
&& j=\ell(\nu)+1,\ell(\nu)+2,\dots,\ell(\nu)+t_2.
\end{align*}
This evaluation turns~\eqref{eq:WLshiftedMaya} evaluated at $x^2/2$ into a Wronskian of Hermite polynomials evaluated in~$x$. In fact, the degrees that appear are precisely all the dots at non-negative locations of a Maya diagram~$\widetilde{M}_\lambda$, which is equivalent to the canonical Maya diagram~$M_\lambda$. Moreover, it is well-known that shifting the origin of a Maya diagram does not change the corresponding Wronskian Hermite polynomial~\cite{Bonneux_Stevens,GomezUllate_Grandati_Milson-durfee}. Then, using the standard Wronskian identity
\begin{equation}\label{eq:WronskianProperty1}
	\Wr[y_1,y_2,\dots,y_r](h(x))
		= \left(h'(x)\right)^{-\frac{r(r-1)}{2}} \Wr[y_1\circ h,y_2\circ h,\dots,y_r\circ h](x)
\end{equation}
for $h(x)=x^2/2$, and  keeping careful track of the factors of $x$ and $2$, we obtain~\eqref{eq:WHermiteWLaguerre}.
\end{proof}

The Wronskian involving Laguerre polynomials~\eqref{eq:WL} appears in the setting of exceptional Laguerre polynomials~\cite{Bonneux_Kuijlaars,Duran-Laguerre,Duran_Perez,GomezUllate_Grandati_Milson-L+J}. In both papers, polynomials $\Omega_{\mu,\nu}^{(\alpha)}$ of degree $|\mu|+|\nu|$ were introduced. Likewise, we define their monic variant as
\begin{equation}\label{eq:omega}
	\hat{\Omega}_{\mu,\nu}^{(\alpha)}
		:=  \frac{(-1)^{\sum_j m_j}}{\Delta(n_{\mu}) \Delta(m_\nu)} e^{-\ell(\nu) x}
		\Wr[f_1,f_2,\dots,f_{\ell(\mu)},h_1,h_2,\dots,h_{\ell(\nu)}]
\end{equation}
for any parameter $\alpha$ with 
\begin{align*}
	f_j(x)
		&= \hat{L}_{n_j}^{(\alpha)}(x),
		&& j=1,2,\dots,\ell(\mu), 
	\\
	h_{j}(x)
		&= e^{x}\hat{L}_{m_j}^{(\alpha)}(-x),
		&& j=1,2,\dots,\ell(\nu).
\end{align*}
Proposition 1 in~\cite{Bonneux_Kuijlaars} states that $\hat{\Omega}_{\mu,\nu}^{(\alpha)}$ is a monic polynomial of degree $|\mu|+|\nu|$ and applying Theorem 1 in~\cite{Bonneux_Kuijlaars} yields the identity
\begin{equation}\label{eq:OmegaWL}
	\hat{L}_{\mu,\nu}^{(\alpha)}(x)
		= \hat{\Omega}_{\mu,\nu'}^{(\alpha-\ell(\nu)-\ell(\nu'))}(x)
\end{equation}
for any partitions $\mu$ and $\nu$ such that the elements in~\eqref{eq:parameter} are pairwise different. More identities between Wronskians involving Laguerre polynomials may be found in~\cite{GomezUllate_Grandati_Milson-L+J}.

Translating the identity in Proposition~\ref{prop:WHermiteWLaguerre} using~\eqref{eq:OmegaWL} gives the following result.

\begin{corollary}\label{cor:WHermiteOmega}
Let $\lambda$ be a partition with core $(k,k-1,\dots,2,1)$ and quotient $(\mu,\nu)$. Then
\begin{equation}\label{eq:ROmega}
	R_{\lambda}(x)	
		= 2^{|\mu|+|\nu|} \, \hat{\Omega}_{\mu,\nu'}^{(\alpha_k)}\left(\frac{x}{2}\right)
\end{equation}
where $\alpha_k:=-1/2-\ell(\mu)-\ell(\nu')-k$ and $\nu'$ denotes the conjugate partition to $\nu$. Similarly,
\begin{equation}\label{eq:WHermiteOmega}
	\He_{\lambda}(x)
		= 2^{|\mu|+|\nu|}  x^{\frac{k(k+1)}{2}} \hat{\Omega}_{\mu,\nu'}^{(\alpha_k)}\left(\frac{x^2}{2}\right).
\end{equation}
\end{corollary}

The Wronskian polynomials~\eqref{eq:WL} and~\eqref{eq:omega} involving Laguerre polynomials were defined above for almost-all values of $\alpha$ and may be obtained for all $\alpha$ by analytical continuation. Hence Proposition~\ref{prop:WHermiteWLaguerre} and Corollary~\ref{cor:WHermiteOmega} connect the discrete parameter $k$ in the Hermite setting to the continuous parameter $\alpha$ in the Laguerre setting.

The following limits involving Laguerre polynomials are well-known: 
\begin{equation}\label{eq:Laguerrelimit}
	\lim_{\alpha\to\pm\infty} \frac{L_n^{(\alpha)}(\alpha x)}{L^{(\alpha)}_n(0)}
		= (1-x)^n
	\quad
	\text{or equivalently}
	\quad
	\lim_{\alpha\to\pm\infty} \frac{\hat{L}_n^{(\alpha)}(\alpha x)}{\alpha^n}
	= (x-1)^n.
\end{equation}
The first limit  can be derived from (5.1.6) in~\cite{Szego}, and the equivalent statement follows from~\eqref{eq:monicLaguerre} and  $L^{(\alpha)}_n(0)=(\alpha+1)(\alpha+2)\cdots(\alpha+n)/n!$. The asymptotic behaviour of the Laguerre polynomials can be generalized to the following asymptotic behaviour of the Wronskian Laguerre polynomials. 

\begin{proposition}\label{prop:asymptoticsWL}
For any pair of partitions $\mu$ and $\nu$ we have
\begin{align}
	&\lim_{\alpha \to \pm\infty} \frac{\hat{L}_{\mu,\nu}^{(\alpha)}(\alpha x)}{\alpha^{|\mu|+|\nu|}}
		= (x-1)^{|\mu|} \, (x+1)^{|\nu|},
	\label{eq:asymptoticsWL} \\
	&\lim_{\alpha \to \pm\infty} \frac{\hat{\Omega}_{\mu,\nu}^{(\alpha)}(\alpha x)}{\alpha^{|\mu|+|\nu|}}
		= (x-1)^{|\mu|} \, (x+1)^{|\nu|}.
	\label{eq:asymptoticsOmega}
\end{align}
\end{proposition}

An heuristic argument for this result can be obtained by combining the asymptotic result~\eqref{eq:asymptotics} for the remainder polynomial and the identities~\eqref{eq:RLaguerre} and~\eqref{eq:ROmega}. However, one should be careful in how the continuous parameter $\alpha$ tends to infinity since only discrete values for~$k$ are permitted. We therefore give a proof of Proposition~\ref{prop:asymptoticsWL} that holds for all values of $\alpha$, without using the connection to Wronskian Hermite polynomials. It is based on the classical result~\eqref{eq:Laguerrelimit} and a careful analysis of determinants. For this, we need the following elementary result for Laguerre polynomials.

\begin{lemma}\label{lem:LaguerreProperty}
Take an integer $m\geq0$ and fix the parameter $\alpha$. Then, for any integer $i\geq1$ we have
\begin{equation*}
\frac{d^{i-1}}{dx^{i-1}} \hat{L}_m^{(\alpha)}(x)= \sum_{l=0}^{i-1} (-1)^l \binom{i-1}{l} \hat{L}_m^{(\alpha+l)}(x).
\end{equation*}
\end{lemma}
\begin{proof}
By induction on $i$. When $i=1$, both sides trivially coincide. Now suppose the claim has been proven for $i$; we prove it for $i+1$. Combining (5.1.13) and (5.1.14) in~\cite{Szego} and translating it via~\eqref{eq:monicLaguerre} yields the identity
\begin{equation}\label{eq:helpfuleqforLaguerre}
	\frac{d}{dx}\left(\hat{L}_{m}^{(\alpha)}(x)\right) 
		= \hat{L}_{m}^{(\alpha)}(x) - \hat{L}_{m}^{(\alpha+1)}(x).
\end{equation}
Taking the $(i-1)^\textrm{st}$ derivative on both sides of~\eqref{eq:helpfuleqforLaguerre}, applying the induction hypothesis for both terms on the right-hand side, as well as some elementary calculations, yields the result for~$i+1$.
\end{proof}

\begin{proof}[Proof of Proposition~\ref{prop:asymptoticsWL}]
It is sufficient to prove~\eqref{eq:asymptoticsOmega} because then~\eqref{eq:asymptoticsWL} follows directly from~\eqref{eq:OmegaWL}. Moreover, we only consider $\alpha\to\infty$ since replacing $\alpha$ by $-\alpha$ and $x$ by $-x$ then gives the result for $\alpha\to-\infty$.

The proof consists of two parts. Firstly, we rewrite the polynomial as a constant times the determinant of a matrix. Secondly, we derive the asymptotic behavior of all entries in this matrix to obtain the limiting behavior of the determinant. 

Consider the Wronskian involving Laguerre polynomials defined in~\eqref{eq:omega}. Using the Wronskian property~\eqref{eq:WronskianProperty1} for $h(x)=\alpha x$ we have
\begin{equation}\label{eq:proofasymptoticsLaguerre1}
	\frac{\hat{\Omega}_{\mu,\nu}^{(\alpha)}(\alpha x)}{\alpha^{|\mu|+|\nu|}} 
		=
		\frac{\kappa}{\alpha^{N}} 
		\,
		e^{-\ell(\nu)\alpha x} 
		\,
		\Wr[\hat{L}_{n_1}^{(\alpha)}(\alpha x),\dots,\hat{L}_{n_{\ell(\mu)}}^{(\alpha)}(\alpha x),e^{\alpha x}\hat{L}_{m_1}^{(\alpha)}(-\alpha x),\dots,e^{\alpha x}\hat{L}_{m_{\ell(\nu)}}^{(\alpha)}(-\alpha x)]
\end{equation}
where 
\begin{equation*}
	\kappa 
		= \frac{(-1)^{\sum\limits_{i=1}^{\ell(\nu)} m_{i}}}{\Delta(n_{\mu})\Delta(m_{\nu})},
	\qquad \qquad
	N
		= \sum\limits_{i=1}^{\ell(\mu)} n_{i}+\sum\limits_{i=1}^{\ell(\nu)} m_{i}+\ell(\mu)\ell(\nu).
\end{equation*}
Next, we distribute the prefactor $\alpha^{-N}$ over the Wronskian entries to write the right-hand side of~\eqref{eq:proofasymptoticsLaguerre1} as 
\begin{equation*}
	\kappa 
	\, 
	e^{-\ell(\nu)\alpha x}
	\,
	\Wr\left[\frac{\hat{L}_{n_1}^{(\alpha)}(\alpha x)}{\alpha^{n_1}},\dots,\frac{\hat{L}_{n_{\ell(\mu)}}^{(\alpha)}(\alpha x)}{\alpha^{n_{\ell(\mu)}}},\frac{e^{\alpha x}\hat{L}_{m_1}^{(\alpha)}(-\alpha x)}{\alpha^{m_1+\ell(\mu)}},\dots,\frac{e^{\alpha x}\hat{L}_{m_{\ell(\nu)}}^{(\alpha)}(-\alpha x)}{\alpha^{m_{\ell(\nu)}+\ell(\mu)}}\right].
\end{equation*}
We can therefore write  
\begin{equation*}
	\frac{\hat{\Omega}_{\mu,\nu}^{(\alpha)}(\alpha x)}{\alpha^{|\mu|+|\nu|}} 
		= \kappa \cdot
		\left|
			\begin{array}{cr}
			A & B \\
			C & D
			\end{array}
		\right|
\end{equation*}
where $A$ is a $\ell(\mu)\times\ell(\mu)$ square matrix and $D$ is a $\ell(\nu)\times\ell(\nu)$ square matrix and where the prefactor $e^{-\ell(\nu)\alpha x}$ is equally distributed over the last $\ell(\nu)$ columns. Explicitly, the entries of the four matrices are given by
\begin{align*}
	A_{ij} 
		&= \frac{d^{i-1}}{dx^{i-1}} \frac{\hat{L}_{n_j}^{(\alpha)}(\alpha x)}{\alpha^{n_j}}, 
		& 1\leq i,j\leq \ell(\mu), \\
	B_{i,j} 
		&= e^{-\alpha x}\frac{d^{i-1}}{dx^{i-1}} \left( \frac{e^{\alpha x} \hat{L}_{m_j}^{(\alpha)}(-\alpha x)}{\alpha^{m_j+\ell(\mu)}} \right), 
		& 1\leq i \leq \ell(\mu), \, 1\leq j \leq \ell(\nu), \\
	C_{ij} 
		&= \frac{d^{\ell(\mu)+i-1}}{dx^{\ell(\mu)+i-1}} \frac{\hat{L}_{n_j}^{(\alpha)}(\alpha x)}{\alpha^{n_j}}, 
		& 1\leq i \leq \ell(\nu), \, 1\leq j\leq \ell(\mu), \\
	D_{i,j} 
		&= e^{-\alpha x}\frac{d^{\ell(\mu)+i-1}}{dx^{\ell(\mu)+i-1}} \left( \frac{e^{\alpha x} \hat{L}_{m_j}^{(\alpha)}(-\alpha x)}{\alpha^{m_j+\ell(\mu)}} \right),
		&  1\leq i,j \leq \ell(\nu). 
\end{align*}
We rewrite the entries $B_{ij}$ and $D_{ij}$ in a more convenient form. Namely,  using~\eqref{eq:helpfuleqforLaguerre}, we have that for all $m\geq 0$
\begin{equation*}
	\frac{d}{dx}\left(e^{\alpha x} \, \hat{L}_m^{(\alpha)}(-\alpha x)\right)
		= \alpha \, e^{\alpha x} \, \hat{L}_m^{(\alpha+1)}(-\alpha x)
\end{equation*}
and using this repeatedly yields 
\begin{equation*}
B_{ij} = \frac{\hat{L}^{(\alpha+i-1)}_{m_j}(-\alpha x)}{\alpha^{m_j+\ell(\mu)-i+1}},
\qquad \qquad
D_{ij} = \frac{\hat{L}^{(\alpha+\ell(\mu)+i-1)}_{m_j}(-\alpha x)}{\alpha^{m_j-i+1}}.
\end{equation*}
We now intend to apply~\eqref{eq:Laguerrelimit} to the individual matrix entries to find the asymptotic behaviour of the determinant as $\alpha$ tends to infinity. The entries of $A$, $B$ and $C$ converge, namely
\begin{align}
\lim_{\alpha\rightarrow \infty} A_{ij} &= \frac{d^{i-1}}{dx^{i-1}}\big( (x-1)^{n_j} \big), \label{eq:limitA}\\
\lim_{\alpha\rightarrow \infty} B_{ij} &= 0,\label{eq:limitB} \\
\lim_{\alpha \rightarrow \infty} C_{ij} &=\frac{d^{\ell(\mu)+i-1}}{dx^{\ell(\mu)+i-1}}\big( (x-1)^{n_j} \big). \label{eq:limitC}
\end{align} 
However, for $i\geq 2$, one sees that the entries of $D_{ij}$ \textit{diverge} as $\alpha \rightarrow \infty$. To counter this, we perform row operations on the rows corresponding to the matrices $C$ and $D$. Namely, we replace $\text{row}\,(i)$ in these matrices with
\[\sum_{l=0}^{i-1} (-1)^l \binom{i-1}{l} \alpha^l \text{ row}\,(i-l).\]
This yields two new matrices $\tilde{C}$ and $\tilde{D}$; for $\tilde{C}$ it is only important that the entries are still convergent as $\alpha \rightarrow \infty$, since they are linear combinations of the (convergent) entries of $C$, see \eqref{eq:limitC}. Furthermore, we have
\begin{equation}\label{eq:omegadet}
	\frac{\hat{\Omega}_{\mu,\nu}^{(\alpha)}(\alpha x)}{\alpha^{|\mu|+|\nu|}} 
		= \kappa \cdot
		\left|
			\begin{array}{cr}
			A & B \\
			\tilde{C} & \tilde{D}
			\end{array}
		\right|
\end{equation}
because row operations do not change a determinant. By Lemma~\ref{lem:LaguerreProperty} we obtain that the entries of $\tilde{D}$ are given by
\begin{equation*}
	\tilde{D}_{i,j}
		= \frac{d^{i-1}}{dx^{i-1}} \left(\frac{\hat{L}_{m_j}^{(\alpha+\ell(\mu))}(-\alpha x)}{\alpha^{m_j}}\right)
\end{equation*}
and hence
\begin{equation}
\label{eq:limittildeD}
	\lim_{\alpha \to\infty} \tilde{D}_{i,j}
		= \frac{d^{i-1}}{dx^{i-1}} \left((-x-1)^{m_j}\right)
\end{equation}
for $1\leq i,j \leq \ell(\nu)$. Now, by~\eqref{eq:limitB} and the fact that the entries of $\tilde{C}$ are convergent, we conclude using~\eqref{eq:omegadet} that
\begin{equation*}
\lim_{\alpha\to\infty} \frac{\hat{\Omega}_{\mu,\nu}^{(\alpha)}(\alpha x)}{\alpha^{|\mu|+|\nu|}}  =  \frac{(-1)^{\sum\limits_{i=1}^{\ell(\nu)} m_{i}}}{\Delta(n_{\mu})\Delta(m_{\nu})} \left( \lim_{\alpha \rightarrow \infty} \det(A)\right) \left( \lim_{\alpha \rightarrow \infty} \det(\tilde{D})\right).
\end{equation*}
By~\eqref{eq:limitA} and~\eqref{eq:limittildeD} we have
\begin{align*}
\lim_{\alpha\to\infty} \det(A) &=\Delta(n_{\mu}) (x-1)^{|\mu|} \\
\lim_{\alpha\to\infty} \det(\tilde{D})&=(-1)^{\sum_j m_{j}} \Delta(m_{\nu}) (x+1)^{|\nu|}
\end{align*}
which then directly implies~\eqref{eq:asymptoticsOmega}, as desired.
\end{proof}

\begin{remark}
As stated before, we have that $\He_{\lambda}$ has no real zeros if and only if $\lambda$ is an even partition, i.e., $\ell(\lambda)$ is even and $\lambda_{2i}=\lambda_{2i-1}$ for all $i=1,2\dots,\ell(\lambda)/2$. Similarly, for the Wronskian involving Laguerre polynomials, see \eqref{eq:WL} or \eqref{eq:OmegaWL}, it is proven in \cite{Duran-Laguerre,Duran_Perez} that there are no zeros on the positive real line if and only if the parameter $\alpha$ and partitions $\mu$ and~$\nu$ satisfy an admissibility condition. Via \eqref{eq:WHermiteOmega} and \eqref{eq:WHermiteWLaguerre}, we can link both polynomials and therefore both conditions should be comparable. Hence a combinatorial interpretation of the admissibility condition in \cite{Duran-Laguerre,Duran_Perez} in terms of quotients seems reasonable, but is omitted here as it does not fit in the scope of this paper.
\end{remark}


\section{\texorpdfstring{Generalization to $p$-cores and $p$-quotients}{Generalization to p-cores and p-quotients}}\label{sec:generalization}
The Hermite polynomials are characterized by  the recurrence relation~\eqref{eq:recurrenceHermite}, or by the exponential generating function 
\begin{equation*}
\sum_{n=0}^\infty \He_n(x) \frac{t^n}{n!} =  \exp\left(tx - \frac{t^2}{2}\right).
\end{equation*}
In the previous sections we showed that the coefficients of Wronskian Hermite polynomials can be understood in terms of the 2-core and 2-quotient of the associated partition. This section is dedicated to showing how these results generalize if one considers the general family of polynomials $(q_n)_{n=0}^\infty$ that  have exponential generating function
\begin{equation}\label{eq:expgenfamily}
	\sum_{n=0}^\infty q_n(x) \frac{t^n}{n!} 
		=  \exp\left(tx - \frac{t^p}{p}\right)
\end{equation}
for an integer $p\geq 2$. Note that we omit the label $p$ in the notation for clarity. These polynomials satisfy the recurrence
\begin{equation}\label{eq:recurrencefamily}
	q_n(x)
		= x q_{n-1}(x) - \frac{(n-1)!}{(n-p)!} q_{n-p}(x)
\end{equation}
for $n\geq p$, with initial conditions $q_n(x)=x^n$ for all $n<p$, and have explicit expansion
\begin{equation}\label{eq:expansionfamily}
	q_n(x) 
		= \sum_{j=0}^{\lfloor n/p \rfloor} (-1)^j \frac{n!}{j! (n-pj)! \,p^j} \,x^{n-pj}
\end{equation} 
for any $n\geq 0$. They are studied in the literature for various reasons. For $p$ odd, the polynomials play a role in the analysis of the rational solutions of the second and higher order Painlev\'{e} equations~\cite[Section 2.7]{Clarkson-survey}. For any $p\geq 2$, the polynomials are also known to be $(p-1)$-orthogonal polynomials~\cite{BenCheikh_Zaghouani} on the \textbf{$p$-star}
\[\bigcup_{l=0}^{p-1} [0,\infty) \times \omega_p^l\]
where $\omega_p$ is the $p^\textrm{th}$ root of unity. Note that in the case that $p=2$, the 2-star is the union of the positive and negative real line, that is the real line itself, which is the domain of orthogonality of the Hermite polynomials. It is well-known that for arbitrary $p$ the zeros of the polynomials~$q_n$ lie on the $p$-star. In fact, the polynomials $q_n$ have the symmetry $q_n(\omega_p \, x) = \omega_p^n \, q_n(x)$, which can either be seen inductively from~\eqref{eq:recurrencefamily} or directly from~\eqref{eq:expansionfamily}. 

Due to the specific form of the exponential generating function~\eqref{eq:expgenfamily}, we know that the sequence $(q_n)_{n=0}^\infty$ is an Appell sequence, that is $q'_n(x)=n q_{n-1}(x)$ for all $n\geq1$. In that context, these polynomial sequences were studied in~\cite[Section~7.2]{Bonneux_Hamaker_Stembridge_Stevens}. In particular, the  polynomials
\begin{equation}\label{eq:Wronskianfamily}
	q_\lambda 
		:= \frac{\Wr[q_{n_1},q_{n_2},\dots,q_{n_{\ell(\lambda)}}]}{\Delta(n_{\lambda})}
\end{equation}
were of interest, analogous to the Wronskian Hermite polynomials~\eqref{eq:WHP}. Amongst others, the specific form of the exponential generating function in~\eqref{eq:expgenfamily} ensures that every $q_\lambda$ has integer coefficients. 

The remainder of this section is dedicated to showing how some of the results from the previous sections for $p=2$ can be generalized to arbitrary $p\geq 2$. We give the necessary notions in Section~\ref{subsec:pcoresquotients} and the results in Section~\ref{subsec:generalizedresults}.

\subsection{\texorpdfstring{$p$-cores and $p$-quotients}{p-cores and p-quotients}}\label{subsec:pcoresquotients}
For our purposes, it is  important that the notion of removing domino tiles for the $p=2$ case should be replaced by removing border strips of size $p$; these are skew Young diagrams that have size $p$, are connected and do not contain any $2\times 2$ squares~\cite{MacDonald,Stanley_EC2}. We write $\gamma\in \mathcal{R}_-^p(\lambda)$ if $\gamma$ is obtained by removing such a border strip of size $p$ from the Young diagram of $\lambda$. For any such border strip, let the height of the border strip $\htt(\lambda/\gamma)$ be the number of rows of the skew Young diagram $\lambda/\gamma$ minus one. More generally, if $\bar{\lambda}$ is obtained from $\lambda$ by removing several border strips of size $p$ consecutively, then $\htt_p(\lambda/\bar{\lambda})$ denotes the sum of the heights of the removed border strips. Even if $\bar{\lambda}$ can be obtained in multiple ways from $\lambda$ by removing border strips, $\htt_p(\lambda/\bar{\lambda})$ is well-defined in this way.

The $p$-core associated to partition $\lambda$ is the partition $\bar{\lambda}$ that is obtained after removing as many border strips of size $p$ as possible; this uniquely defines the $p$-core of a partition. Equivalently, one can define the $p$-cores as all partitions whose hook lengths in the Young diagram are all not divisible by~$p$. 

On the other hand, the $p$-quotient is an ordered tuple $\mu=(\mu^0,\mu^{1},\dots,\mu^{p-1})$ of partitions. The set of all $p$-quotients forms the $p$-fold product lattice $\mathbb{Y}^p$ with natural ordering $\leq$ inherited from the product. The size of the  $p$-quotient is naturally defined as $\lvert \mu \rvert =\sum_i \lvert \mu^{i} \rvert$, and for any tuples $\mu,\tilde{\mu}\in \mathbb{Y}^p$ and integer $j\geq0$, we write $\tilde{\mu} <_j \mu$ if and only if $\tilde{\mu} \leq \mu$ and $\lvert \tilde{\mu} \rvert +j =\lvert \mu \rvert$; if $j=1$, we sometimes write $\tilde{\mu}\lessdot \mu$ instead of $<_1$.

For $\mu,\nu \in \mathbb{Y}^p$, the number of lattice paths from $\nu$ to $\mu$ is denoted by $F_{\mu/\nu}^{(p)}$ and this number is equal to
\begin{equation*}
	F_{\mu/\nu}^{(p)} 
		= \binom{\lvert \mu \rvert - \lvert \nu \rvert}{\lvert \mu^0 \rvert - \lvert \nu^0 \rvert, \lvert \mu^1 \rvert - \lvert \nu^1 \rvert, \dots, \lvert \mu^{p-1} \rvert - \lvert \nu^{p-1} \rvert} \prod_{i=0}^{p-1} F_{\mu^{i}/\nu^{i}}.
\end{equation*}
We write $F_{\mu}^{(p)}$ instead of $F_{\mu/\emptyset}^{(p)}$ when $\nu=\emptyset$ is the tuple of empty  partitions. 

The essential link between the notion of border strips and the $p$-quotient is the fact that~$\tilde{\lambda}$ is obtained by removing a border strip of size~$p$ from~$\lambda$ if and only if the~$p$-quotient $\tilde{\mu}$ of $\tilde{\lambda}$ and the $p$-quotient $\mu$ of $\lambda$ satisfy $\tilde{\mu} \lessdot \mu$. For the precise definitions of $p$-cores and $p$-quotients we refer to~\cite[I.1~Ex.~ 8]{MacDonald}. We note that the ordering of the partitions in a $p$-quotient is defined up to cyclic transformations. For the case $p=2$ we required that $k\geq 0$, which uniquely defines the precise order in the quotient. For a full-fledged generalization of the results in this paper to the case of general $p\geq 2$, one similarly needs to fix the ordering of the quotient. Nevertheless, in this section we only generalize a subset of the results and the ordering is unimportant for these results. Moreover, we note that just as for the construction of the 2-quotient described in Section~\ref{sec:coresandquotients} when $p=2$, the construction of the $p$-quotient is equivalent to the $p$-modular decomposition of Maya diagrams given in~\cite{Clarkson_GomezUllate_Grandati_Milson}.

\begin{remark}
The rational solutions of the fourth Painlev\'e equation are defined in terms of Wronskian Hermite polynomials. The partitions that label these polynomials belong to two separate classes. On the one hand one has the rectangular partitions $(m^n)$, which give rise to the so-called generalized Hermite polynomials \cite{Buckingham,Clarkson-zeros,Masoero_Roffelsen-2019}. On the other hand one has the class of partitions that give rise to the generalized Okamoto polynomials \cite{Clarkson-PIV,Kajiwara_Ohta,Novokshenov_Shchelkonogov,Noumi_Yamada,Masoero_Roffelsen}; these partitions are of the form
\begin{align*}
	&\lambda =(m+2n,m+2n-2,\dots,m+2,m,m,m-1,m-1,\dots,1,1) \text{ or }\\
	&\lambda =(m+2n-1,m+2n-3,\dots,m+1,m,m,m-1,m-1,\dots,1,1).
\end{align*}
These partitions are precisely the 3-cores~\cite{Noumi}. This observation can also be drawn when interpreting the results in \cite{Clarkson_GomezUllate_Grandati_Milson} in terms of partitions.
\end{remark} 

\subsection{Generalized results}\label{subsec:generalizedresults}
Most of the results from the main section for the Hermite $p=2$ case generalize to arbitrary $p\geq 2$. We omit the proofs of these results, since they follow directly from generalizing the proofs in the previous sections. This is all based on the fact that the Wronskian polynomials defined in~\eqref{eq:Wronskianfamily}  satisfy the generating recurrence
\begin{equation}\label{eq:GRRfamily}
	F_\lambda q_\lambda 
		=\frac{x}{\lvert \lambda \rvert} F_\lambda q_\lambda' + \frac{(\lvert \lambda \rvert -1)!}{(\lvert \lambda \rvert -r)!} \sum_{\tilde{\mu} \lessdot \mu} (-1)^{\htt(\lambda/\tilde{\lambda})} F_{\tilde{\lambda}} q_{\tilde{\lambda}}
\end{equation}
with $\mu\in\mathbb{Y}^p$ being the $p$-quotient of $\lambda$ and $\tilde{\mu}\in\mathbb{Y}^p$ the $p$-quotient of $\tilde{\lambda}$. This result follows from interpreting the generating recurrence from~\cite[Section~7.2]{Bonneux_Hamaker_Stembridge_Stevens} in terms of cores and quotients. In particular, this implies that one has a similar factorization as that given in Theorem~\ref{thm:WHdecomposition}. For this, we generalize the notion of $H_{\odd}(\lambda)$ when $p=2$, which is the product of all odd hook lengths in the Young diagram of $\lambda$, to the notion of $H_{\text{non-$p$-fold}}(\lambda)$, which is the product of all hook lengths that are not a multiple of $p$.

\begin{theorem}\label{thm:factorizationfamily}
For any $p\geq 2$ and any partition $\lambda$ with $p$-core $\bar{\lambda}$ and $p$-quotient $\mu \in \mathbb{Y}^p$ we have
\begin{equation}\label{eq:factorizationfamily}
	q_\lambda(x) =
		x^{|\bar{\lambda}|} R_{\lambda}(x^p)
\end{equation}
where $R_\lambda$ is a monic polynomial of degree $|\mu|$ with non-vanishing constant coefficient
\begin{equation*}
	R_{\lambda}(0)
		= (-1)^{h_{\lambda}} \frac{H_{\text{non-$p$-fold}}(\lambda)}{H(\bar{\lambda})}
\end{equation*}
where $h_{\lambda}=\htt_p(\lambda/\bar{\lambda})+(|\lambda|-|\bar{\lambda}|)/p$.
\end{theorem}

Again, we call $R_\lambda$ the remainder polynomial and we omit the dependency on $p$ from the notation. As in Section~\ref{sec:2coresandWHP}, we have the following two corollaries. The first one uses the fact that $q_{\lambda}(x)\in\mathbb{Z}[x]$ for any partition $\lambda$, while the other one is based on the identity ${q_{\lambda}(x) = \omega_p^{|\lambda|} q_{\lambda'}(\omega_p^{-1}x)}$ for ${\omega_p=-\exp(i\pi/p)}$; see~\cite[Section~7.2]{Bonneux_Hamaker_Stembridge_Stevens}.

\begin{corollary}\label{cor:intp}
For any integer $p\geq2$ and for any partition $\lambda$ with $p$-core $\bar{\lambda}$ we have 
$$
\frac{H_{\text{non-$p$-fold}}(\lambda)}{H(\bar{\lambda})} \in \mathbb{Z}.
$$
\end{corollary}

\begin{corollary}
For any $p\geq 2$ and any partition $\lambda$ with $p$-quotient $\mu\in\mathbb{Y}^p$ we have 
\begin{equation*}
	R_{\lambda}(x)
		= (-1)^{(p-1)|\mu|} \, R_{\lambda'}\big((-1)^{p-1} x\big)
\end{equation*}
where $\lambda'$ denotes the conjugate partition to $\lambda$.
\end{corollary}

We expand the remainder polynomial as
\begin{equation}\label{eq:expansionremainderfamily}
	R_\lambda(x)
		=\sum_{j=0}^{\lvert \mu \rvert} r_{\lambda,j} \, x^{\lvert \mu \rvert -j}
\end{equation}
and, by~\eqref{eq:GRRfamily}, the recurrence for the coefficients becomes
\begin{equation}\label{eq:rp}
	F_{\lambda} r_{\lambda,j} 
		= - \frac{\lvert \lambda \rvert!}{pj (\lvert \lambda \rvert -p)!} \sum_{\tilde{\mu} \lessdot \mu} (-1)^{\htt_p(\lambda/\tilde{\lambda})} F_{\tilde{\lambda}} \,r_{\tilde{\lambda},j-1}
\end{equation}
where again $\tilde{\mu}\in\mathbb{Y}^p$ and $\mu\in\mathbb{Y}^p$ are the $p$-quotients of $\tilde{\lambda}$ and $\lambda$, respectively, and $j=0,1,\dots,|\mu|$. The relation \eqref{eq:rp} generates all coefficients of the remainder polynomial if one takes into account that $r_{\lambda,0}=1$ for all $\lambda$. In fact, applying \eqref{eq:rp} recursively~$j$ times yields the explicit expansion for the Wronskian polynomials  stated in the following theorem. It should be compared to~\eqref{eq:expansionfamily} and it naturally generalizes Theorem~\ref{thm:WHcoeff} using the character values of cycle type $(p^j, 1^{\lvert \lambda \rvert-pj})$.

\begin{theorem}\label{thm:coefffamily}
For any partition $\lambda$ we have
\begin{equation}\label{eq:coefffamily}
	F_\lambda q_{\lambda}(x)
		=\sum_{j=0}^{\lfloor\lvert \lambda \rvert/p \rfloor} (-1)^j \frac{\lvert \lambda \rvert!}{j! (\lvert \lambda \rvert-pj)! p^j} \,a_p(\lambda,j) \,x^{\lvert \lambda \rvert-pj}
\end{equation}
with $a_p(\lambda,j)$ being the character of the conjugacy class of the cycle type $(p^j,1^{\lvert \lambda \rvert-pj})$ of the irreducible representation associated to the partition $\lambda$ of the symmetric group $S_{|\lambda|}$. 
\end{theorem} 

\begin{remark}
In~\cite{Bonneux_Hamaker_Stembridge_Stevens}, it was shown that the average of the Wronskian polynomials with respect to the Plancherel measure is simply the monomial; that is
\[\sum_{\lambda \vdash n} \frac{F_\lambda^2}{n!} q_\lambda(x) = x^n.\]
As in Remark~\ref{rem:character}, this result is also equivalent to the orthogonality of characters, but now those evaluated in cycle types $(p^j,1^{n-pj})$.
\end{remark}

Theorem~\ref{thm:coefficients} also has an analogue for general $p\geq2$.

\begin{theorem}\label{thm:coefficientsfamily}
Let $\lambda$ be a partition with $p$-quotient $\mu \in \mathbb{Y}^p$. Then the coefficients of the remainder polynomial $R_\lambda$, defined in~\eqref{eq:factorizationfamily} and~\eqref{eq:expansionremainderfamily}, are given by
\begin{equation}\label{eq:coefficientsfamily}
	r_{\lambda,j}
	= (-1)^j \binom{|\mu|}{j} 
	\sum_{\tilde{\mu}<_j \mu} (-1)^{\htt_p(\lambda/\tilde{\lambda})}
	\frac{F^{(p)}_{\tilde{\mu}} \,F^{(p)}_{\mu/\tilde{\mu}}}{F^{(p)}_{\mu}} 
	\,
	\frac{H_{\text{non-$p$-fold}}(\lambda)}{H_{\text{non-$p$-fold}}(\tilde{\lambda})}
\end{equation}
for $j=0,1,\dots,|\mu|$, and where the partition $\tilde{\lambda}$ has $p$-quotient $\tilde{\mu}\in\mathbb{Y}^p$ and the same $p$-core as $\lambda$.
\end{theorem}

\begin{remark}
In this section, we have shown how the results from the previous section extend from $p=2$ to $p\geq 2$. However, for $p=1$, all the results trivialize. Namely, if $p=1$, the exponential generating function~\eqref{eq:expgenfamily} defines $q_n(x)=(x-1)^n$ for all $n\geq0$. This then leads to the fact that $q_\lambda(x) =(x-1)^{\lvert \lambda \rvert}$ for all $\lambda$. Everything trivializes since the 1-core of every partition is the empty partition, and the 1-quotient is the partition itself. For example, the coefficients in~\eqref{eq:coefficientsfamily} are the coefficients of the binomial expansion of $(x-1)^{|\lambda|}$ when $p=1$.
\end{remark}


\section{Conclusion and further research}
In this paper we have given a combinatorial interpretation for the coefficients of Wronskian Hermite polynomials with the core and quotient representation of a partition as the main ingredient. The framework elaborates the fact that the polynomial properties are directly related to aspects of the associated partition. We believe that this is further evidence that the use of partitions is the most convenient and elegant way to treat these polynomials,  and their natural generalizations.

An open problem for Wronskian Hermite polynomials concerns the multiplicity of the zeros. Veselov conjectured that all zeros not at the origin must be simple~\cite{Felder_Hemery_Veselov}, and it was known that the multiplicity at the origin is a triangular number. The latter statement is now proven by Theorem~\ref{thm:WHdecomposition}. Moreover, Veselov's conjecture is now equivalent to stating that the remainder polynomial $R_\lambda$ has simple zeros. A possible way of approaching the conjecture is by showing that the \textit{discriminant} of $R_\lambda$ is always non-zero. We believe that it is natural to study this question in terms of cores and quotients. According to Theorem~\ref{thm:coeffpoly}, we have that for a fixed quotient $(\mu,\nu)$, the coefficients of $\He_{\Phi(\mu,\nu,k)}$ are polynomials in~$k$, where $k$ is the length of the core. This means that for a fixed quotient, the discriminant of $R_{\Phi(\mu,\nu,k)}$ is also polynomial in~$k$, and therefore has finitely many zeros. Numerical evidence suggests that the values of~$k$ where the discriminant of $R_{\Phi(\mu,\nu,k)}$ is zero are non-integer values, and so the truth of this statement for every quotient $(\mu,\nu)$ would prove Veselov's conjecture. In fact, we believe that the values of~$k$ where the discriminant of $R_{\Phi(\mu,\nu,k)}$ is zero precisely coincide with the values of $k$ where the Wronskian involving Laguerre polynomial~\eqref{eq:WL} has non-simple zeros, based on Proposition~\ref{prop:WHermiteWLaguerre} and Corollary~\ref{cor:WHermiteOmega}. A possible starting point is the work of~\cite{Roberts} where the author derived explicit expressions for the discriminant of Yablonskii-Vorobiev and other polynomials related to rational solutions of Painlev\'e equations. However,  these solutions are always associated to specific choices of partitions and the method in~\cite{Roberts} does not naively  extend to all Wronskian Hermite polynomials. 

In Section~\ref{sec:generalization}, we showed that many of the results that hold for Wronskian Hermite polynomial extend to the Wronskian polynomials associated to polynomial sequences satisfying~\eqref{eq:expgenfamily}. However, we have not generalized all results available for the Hermite $p=2$ case. In principle, this is due to the fact that a description of a partition in terms of a 2-core and a 2-quotient is significantly easier than the description in terms of a $p$-core and a $p$-quotient. First of all, a 2-core is described by the one parameter $k$ that we used throughout this article; a $p$-core depends on $p-1$ parameters, that one might suitably label $k_1,k_2,\dots,k_{p-1}$. At this moment, it is unclear how to do this precisely. Secondly, as mentioned in Section~\ref{sec:generalization}, the ordering of the quotient is usually only defined up to cyclic transformation. For $p=2$, we fixed the ordering by requiring that $k\geq 0$. Many of the results we have for Wronskian Hermite polynomials depend on the specific ordering of the 2-quotient $(\mu,\nu)$; for example, the asymptotic result in Theorem~\ref{thm:asymptotics} is not symmetric in interchanging $\mu$ and $\nu$. Therefore, it is expected that generalizations of these results also depend on the way the ordering is fixed. Writing down the natural way of doing this in this context is part of our future work.

Another related research problem is to describe  the exact location of the zeros of the Wronksian Hermite polynomials. In Section~\ref{sec:asymptoticzeros}, we observed that when the size of the core is sufficiently large, then the location of the zeros of the remainder polynomial  are related to the quotient partitions in the way shown in Figure~\ref{fig:AsymptoticBehaviorExamplebis}. More detailed studies should help us to gain an understanding of how the zeros behave as the core size increases. Again, we anticipate that cores, quotients and Maya diagrams of partitions will play a key r\^ole in this aspect of the story. 

According to Theorem~\ref{thm:WHcoeff}, the coefficients of Wronskian Hermite polynomials are connected to the characters of irreducible representations of the symmetric group. Specifically, one needs the characters evaluated in cycle types of the form $(2^j,1^{n-2j})$. In Theorem~\ref{thm:coefffamily} we showed how the characters evaluated in the cycle types $(p^j,1^{n-pj})$ appear in Wronskian Appell polynomials. A natural question to ask is whether Wronskians of other polynomials give information about the characters evaluated in the remaining cycle types.


\begin{appendices}
\section{Appendix}\label{sec:appendix}

\subsection{Combinatorial identity}
The following identity is used in the proof of Property~\ref{prop:subleadingcoeff}.

\begin{lemma}\label{lem:identity}
Let $x_1,x_2,\dots,x_n$ be $n$ pairwise different complex numbers, then 
\begin{equation}\label{eq:identity}
	\sum_{j=1}^{n} x_j (x_j-1)  \prod_{i \neq j}\frac{x_j-x_i-2}{x_j-x_i}
		= \sum_{j=1}^{n}(x_j-n+j)(x_j-n-j+1).
\end{equation}
\end{lemma}
\begin{proof}
As a preliminary step we find that expanding the right-hand side of~\eqref{eq:identity} gives
\begin{equation}\label{eq:NumbersIdentity3Proof1}
	\sum_{j=1}^{n}(x_j-n+j)(x_j-n-j+1)
		= \sum_{j=1}^{n}\left(x_j^2+(1-2n)x_j\right) + \frac{2}{3}n^3-n^2+\frac{1}{3}n.
\end{equation}
We approach by induction on $n$ and show that the left-hand side of~\eqref{eq:identity} equals the right-hand side of~\eqref{eq:NumbersIdentity3Proof1}. For $n=1$ the equality is trivial and so we take $n>1$. We claim that the left-hand side of~\eqref{eq:identity} equals
\begin{equation}\label{eq:NumbersIdentity3Proof2}
	x_n^2+(1-2n)x_n + 2 (n-1)^2-2\sum_{i=1}^{n-1}x_i
		+\sum_{j=1}^{n-1} x_j (x_j-1)  \prod_{\substack{i=1 \\ i \neq j}}^{n-1}\frac{x_j-x_i-2}{x_j-x_i}		
\end{equation}
such that the required result follows if we apply the induction hypothesis to the last term of~\eqref{eq:NumbersIdentity3Proof2}. So we only need to prove this claim.
	
Extracting the last term of the sum in the left-hand side of~\eqref{eq:identity} gives
\begin{equation}\label{eq:NumbersIdentity3Proof3}
	\sum_{j=1}^{n-1}x_j(x_j-1) \prod_{\substack{i=1 \\ i\neq j}}^{n}\frac{x_j-x_i-2}{x_j-x_i}
	+ x_n(x_n-1) \prod_{i=1}^{n-1}\frac{x_n-x_i-2}{x_n-x_i}.
\end{equation}
We now consider $x_n$ as a variable and derive the partial fractal decomposition of the last term in~\eqref{eq:NumbersIdentity3Proof3} in the form 
\begin{equation}\label{eq:parfrac}
	x_n(x_n-1) \prod_{i=1}^{n-1}\frac{x_n-x_i-2}{x_n-x_i}
		= B x_n^2 + C x_n + D + \sum_{j=1}^{n-1} \frac{A_j}{x_n-x_j}
\end{equation}
for some constants $B,C$ and $D$, where
\begin{equation}\label{eq:Aj}
	A_j 
		= x_j(x_j-1) \frac{\prod\limits_{i=1}^{n-1}x_j-x_i-2}{\prod\limits_{\substack{i=1 \\ i\neq j}}^{n-1}x_j-x_i}
		= -2 x_j(x_j-1) \prod_{\substack{i=1 \\ i\neq j}}^{n-1}\frac{x_j-x_i-2}{x_j-x_i}
\end{equation}
for $j=1,2,\dots,n-1$. Adding the term $i=n$ to the last product of~\eqref{eq:Aj} and using the equality
\begin{equation*}
	\frac{-2}{x_j-x_n-2}= 1 - \frac{x_j-x_n}{x_j-x_n-2}
\end{equation*}
yields
\begin{equation*}
	A_j 
		=(x_j-x_n) x_j(x_j-1) \left(\prod_{\substack{i=1 \\ i\neq j}}^{n}\frac{x_j-x_i-2}{x_j-x_i} - \prod_{\substack{i=1 \\ i\neq j}}^{n-1}\frac{x_j-x_i-2}{x_j-x_i}\right)
\end{equation*}
for $ j=1,2,\dots,n-1$. To derive the constants $B,C$ and $D$ in~\eqref{eq:parfrac}, we expand the product
\begin{equation*}
	\prod_{i=1}^{n-1}\frac{x_n-x_i-2}{x_n-x_i}
		= 1-2\sum_{i=1}^{n-1}\frac{1}{x_n-x_i} +4  \sum_{1\leq i < j \leq n-1} \frac{1}{(x_n-x_i)(x_n-x_j)} + O\left(x_n^{-3}\right)
\end{equation*}
such that after some elementary calculations we find that
\begin{equation*}
	B=1,
	\qquad 
	C=1-2n,
	\qquad
	D= 2 (n-1)^2-2\sum_{i=1}^{n-1}x_i.
\end{equation*}
Therefore we  have that the partial fraction decomposition is given by
\begin{multline}\label{eq:NumbersIdentity3Proof4}
	x_n(x_n-1) \prod_{i=1}^{n-1}\frac{x_n-x_i-2}{x_n-x_i}
	= x_n^2 + \left(1-2n\right)x_n
	+ 2 (n-1)^2-2\sum_{i=1}^{n-1}x_i
	\\
	+\sum_{j=1}^{n-1}x_j(x_j-1) \prod_{\substack{i=1 \\ i\neq j}}^{n-1}\frac{x_j-x_i-2}{x_j-x_i}
	- \sum_{j=1}^{n-1}x_j(x_j-1) \prod_{\substack{i=1 \\ i\neq j}}^{n}\frac{x_j-x_i-2}{x_j-x_i}. 
\end{multline}
Hence the claim follows from plugging~\eqref{eq:NumbersIdentity3Proof4} into~\eqref{eq:NumbersIdentity3Proof3} to obtain~\eqref{eq:NumbersIdentity3Proof2}. 
\end{proof}

\subsection{Leading coefficients of Wronskian Appell polynomials}\label{sec:leadingcoefficientsWAP}

In this section we extend the result of Proposition~\ref{prop:subleadingcoeff} to Wronskian Appell polynomials \cite{Bonneux_Hamaker_Stembridge_Stevens}. An Appell sequence $(A_n)_{n=0}^\infty$ is a polynomial sequence satisfying $A_n'=nA_n$ for all $n\geq 1$ and~$A_0\equiv 1$. From this definition it immediately follows that there is a set of constants $(z_j)_{j=0}^\infty$ with $z_0=1$ such that
\begin{equation*}
	A_n(x)
		=\sum_{j=0}^n \binom{n}{j} z_j \,x^{n-j}
\end{equation*}
for all $n\geq0$. In particular, $A_n(0)=z_n$. Wronskian Appell polynomials are then defined by
\begin{equation*}
	A_{\lambda}(x)
		:= \frac{\Wr[A_{n_1},A_{n_2},\dots,A_{n_{\ell(\lambda)}}]}{\Delta(n_{\lambda})} 
\end{equation*}
for any partition $\lambda$. This is in line with the definition of Wronskian Hermite polynomials~\eqref{eq:WHP}. In the Hermite setting we have $z_1=0$ and $z_2=-1$. We give an explicit expression for the first three coefficients of Wronskian Appell polynomials in terms of the above constants $z_1$ and~$z_2$, which agree with the results for Wronskian Hermite polynomials in Section~\ref{sec:coefficients} when we set $z_1=0$ and $z_2=-1$. 

\begin{proposition}
For any partition $\lambda\vdash n$ we have
\begin{equation}
\label{eq:firstcoefficientsWAP}
	A_{\lambda}(x)
		= x^{n} + \binom{n}{1} z_1 x^{n-1} + \left(c(\lambda) (z_2-z_1^2) + \binom{n}{2} z_1^2\right) x^{n-2} + O(x^{n-3})
\end{equation}
where the content 
$c(\lambda)$ is defined in~\eqref{eq:contents}.
\end{proposition}
\begin{proof}
For any partition $\lambda\vdash n$ we denote the subleading and subsubleading coefficients of $A_\lambda$ by $a_{\lambda,1}$ and $a_{\lambda,2}$, respectively. In other words, we have
\begin{equation}
\label{eq:firstexpWAP}
A_{\lambda}(x):= x^{n} + a_{\lambda,1} x^{n-1} + a_{\lambda,2} x^{n-2} + O(x^{n-3})
\end{equation}
and we recall that $A_{\lambda}$ is a monic polynomial by definition. Trivially, the derivative is 
\begin{equation}
\label{eq:firstexpWAPprime}
A'_{\lambda}(x)= n x^{n-1} + (n-1) a_{\lambda,1} x^{n-2} + (n-2) a_{\lambda,2} x^{n-3} + O(x^{n-4}).
\end{equation}
We now prove \eqref{eq:firstcoefficientsWAP} by induction on $n:=|\lambda|$. A simple calculation from the definition gives 
\begin{align*}
	& A_\emptyset(x)=1 &&A_{(1)}(x)=x+z_1 \\
	& A_{(2)}(x)=x^2+2z_1x+z_2 && A_{(1,1)}(x)=x^2+2z_1x+2z_1^2-z_2
\end{align*}
and so \eqref{eq:firstcoefficientsWAP} holds for $n\leq2$. Therefore take $n>2$, assume that the identity is true for all partitions $\mu$ such that $|\mu|<n$, and let $\lambda \vdash n$. The derivative of the Wronskian Appell polynomial is given by
\begin{equation}
\label{eq:aderiv}
	F_{\lambda} A'_{\lambda}(x)
		= n \sum_{\mu \lessdot \lambda} F_{\mu} A_{\mu}(x)
\end{equation}
for all $\lambda$, see~\cite[Theorem~5.1]{Bonneux_Hamaker_Stembridge_Stevens}. We plug \eqref{eq:firstexpWAPprime} into the left-hand side of \eqref{eq:aderiv} and \eqref{eq:firstexpWAP} for every~$\mu$ on the right-hand side of \eqref{eq:aderiv}, then match the coefficients of the leading terms to obtain
\begin{equation*}
	F_{\lambda} (n-1) \,a_{\lambda,1}
		= n \sum_{\mu \lessdot \lambda} F_{\mu} \, a_{\mu,1},
	\qquad \qquad
	F_{\lambda} (n-2) \,a_{\lambda,2}
		= n \sum_{\mu \lessdot \lambda} F_{\mu} \, a_{\mu,2}.
\end{equation*}
We now apply the induction hypothesis, which says that for every $\mu\lessdot \lambda$ we have
\begin{equation*}
	a_{\mu,1}=(n-1) z_1
	\qquad \qquad
	a_{\mu,2}= c(\mu)(z_2-z_1^2) + \binom{n-1}{2}z_1^2
\end{equation*}
to obtain 
\begin{equation*}
	F_{\lambda} a_{\lambda,1}
		= n z_1 \sum_{\mu \lessdot \lambda} F_{\mu},
	\qquad \qquad
	F_{\lambda} a_{\lambda,2}
		= \frac{n}{n-2} (z_2-z_1^2) \sum_{\mu \lessdot \lambda} F_{\mu} \, c(\mu) + \binom{n}{2}z_1^2 \sum_{\mu \lessdot \lambda} F_{\mu}.
\end{equation*}
The result follows immediately from the combinatorial identities
\begin{equation}
\label{eq:combinatorialidentities}
	\sum_{\mu \lessdot \lambda} F_{\mu} = F_{\lambda},
	\qquad \qquad
	\sum_{\mu \lessdot \lambda} F_{\mu} \, c(\mu) = \frac{n-2}{n} \, F_{\lambda} \, c(\lambda).
\end{equation}
The first identity is trivial while the second one is proven in Corollary~\ref{cor:FlambdaClambda}.
\end{proof}

All that is now left to prove is the second identity in \eqref{eq:combinatorialidentities}. This identity follows from the following result.

\begin{lemma}\label{lem:combinatorialidentity}
For any partition $\lambda$ we have
\begin{equation}\label{eq:combinatorialidentity}
	F_{\lambda} \, c(\lambda)
		= \binom{|\lambda|}{2} (F_{\lambda/(2)}-F_{\lambda/(1,1)})
\end{equation}
where $c(\lambda)$ is defined in~\eqref{eq:contents}.
\end{lemma}
\begin{proof}
For any partition $\lambda$, the associated Wronskian Appell polynomial is given by
\begin{equation*}
	F_\lambda A_{\lambda}(x)	
		= \sum_{j=0}^{|\lambda|} \binom{|\lambda|}{j} \sum_{\mu\vdash j} F_{\mu} \, F_{\lambda/\mu} \, z_{\mu} \, x^{|\lambda|-j}
\end{equation*}
where $z_\mu:=A_{\mu}(0)$, see \cite[Section 5.1]{Bonneux_Hamaker_Stembridge_Stevens}. For Hermite polynomials we have $z_1=0$ and $z_2=-1$. If we therefore specify to Wronskian Hermite polynomials, the coefficient corresponding to $x^{|\lambda|-2}$ equals
\begin{equation*}
	\binom{|\lambda|}{2} \left( -F_{\lambda/(2)} +  F_{\lambda/(1,1)} \right)
\end{equation*}
because $z_{(2)}=-1$ and $z_{(1,1)}=1$. In terms of the remainder polynomial, this coefficient is by definition equal to~$F_\lambda r_{\lambda,1}$. However, by Proposition~\ref{prop:subleadingcoeff} we know that $r_{\lambda,1}=-c(\lambda)$. Combining both expressions yields the result.
\end{proof}

\begin{corollary}\label{cor:FlambdaClambda}
For any partition $\lambda$ we have
\begin{equation*}
	|\lambda| \sum_{\mu \lessdot \lambda} F_{\mu} \, c(\mu) 
		= (|\lambda|-2) \, F_{\lambda} \, c(\lambda)
\end{equation*}
where $c(\lambda)$ is defined in~\eqref{eq:contents}.
\end{corollary}
\begin{proof}
Applying~\eqref{eq:combinatorialidentity} to all partitions $\mu \lessdot \lambda$ yields
\begin{equation*}
	|\lambda| \sum_{\mu \lessdot \lambda} F_{\mu} \, c(\mu) 
		= |\lambda| \binom{|\lambda|-1}{2} \sum_{\mu \lessdot \lambda} (F_{\mu/(2)}-F_{\mu/(1,1)}).
\end{equation*}
Since trivially $\sum_{\mu \lessdot \lambda} F_{\mu/\nu} = F_{\lambda /\nu}$ for any $\nu$, we therefore obtain
\begin{equation*}
	|\lambda|\sum_{\mu \lessdot \lambda} F_{\mu} \, c(\mu)
		=|\lambda| \binom{|\lambda|-1}{2} (F_{\lambda/(2)}-F_{\lambda/(1,1)}) 
		=(|\lambda|-2) \, F_{\lambda} \, c(\lambda)
\end{equation*}
where to obtain the last equality we again used~\eqref{eq:combinatorialidentity}.
\end{proof}


\end{appendices}

\section*{Acknowledgements}
We would like to thank Roger Behrend, Chris Bowman, Peter Clarkson, Matthew Fayers, Arno Kuijlaars, Ana Loureiro, Rowena Paget, Walter Van Assche and Mark Wildon  for useful conversations. NB thanks the University of Kent, and TCD thanks Australia National University and its Mathematical Sciences Research Visitor Program and KU Leuven for hospitality during this project. The project was partially supported by London Mathematical Society  Research in Pairs grant 41848 and by the London Mathematical Society Joint Research group Orthogonal Polynomials, Special Functions and Operator Theory and Applications. NB and MS are supported in part by the long term structural funding-Methusalem grant of the Flemish Government, and by EOS project 30889451 of the Flemish Science Foundation (FWO). MS is also supported by the Belgian Interuniversity Attraction Pole P07/18, and by FWO research grant G.0864.16.

\end{document}